\documentclass[11pt]{article}

\usepackage{amsmath,amssymb,amsthm}
\usepackage{aliascnt}
\usepackage{mathtools}
\usepackage{booktabs}
\usepackage{array}
\usepackage{enumitem}
\usepackage[margin=1in]{geometry}
\usepackage{microtype}
\usepackage{listings}
\usepackage{tikz}
\usetikzlibrary{arrows.meta}

\usepackage[
  backend=biber,
  style=authoryear,
  maxcitenames=2,
  maxbibnames=99,
  giveninits=true,
  uniquename=init,
  doi=true,
  isbn=true,
  url=false
]{biblatex}
\usepackage[
  colorlinks=true,
  linkcolor={blue!55!black},
  citecolor={green!45!black},
  urlcolor={blue!75!black}
]{hyperref}
\hypersetup{
  pdftitle={Sylvester's Theorem and Reduced Linear Systems: From Three-Factor Products to Small Dynamical Models},
  pdfauthor={James M. Hyman},
  pdfsubject={Lecture notes in linear algebra and numerical analysis},
  pdfkeywords={Sylvester's determinant identity, low-rank factorization, reduced linear systems,
    projection-based model reduction, singular values, lecture notes}
}
\usepackage[nameinlink,noabbrev]{cleveref}

\newcommand{\F}{\mathbb{F}}
\newcommand{\R}{\mathbb{R}}
\newcommand{\C}{\mathbb{C}}
\newcommand{\I}{\mathbf{I}}
\newcommand{\mat}[1]{\mathbf{#1}}
\newcommand{\vect}[1]{\mathbf{#1}}
\newcommand{\rank}{\operatorname{rank}}
\newcommand{\range}{\operatorname{range}}
\newcommand{\spec}{\operatorname{spec}}
\newcommand{\nullsp}{\operatorname{null}}

\newcommand{\RePart}{\operatorname{Re}}
\newcommand{\e}{\mathrm{e}}
\newcommand{\dd}{\mathrm{d}}
\newcommand{\norm}[1]{\left\lVert #1\right\rVert}

\newtheorem{theorem}{Theorem}[section]
\newaliascnt{proposition}{theorem}
\newtheorem{proposition}[proposition]{Proposition}
\aliascntresetthe{proposition}
\newaliascnt{corollary}{theorem}
\newtheorem{corollary}[corollary]{Corollary}
\aliascntresetthe{corollary}
\newaliascnt{lemma}{theorem}

\aliascntresetthe{lemma}
\theoremstyle{definition}
\newaliascnt{definition}{theorem}
\newtheorem{definition}[definition]{Definition}
\aliascntresetthe{definition}
\newaliascnt{example}{theorem}
\newtheorem{example}[example]{Example}
\aliascntresetthe{example}
\newtheorem{exercise}{Exercise}[section]
\theoremstyle{remark}
\newaliascnt{remark}{theorem}
\newtheorem{remark}[remark]{Remark}
\aliascntresetthe{remark}

\crefname{theorem}{theorem}{theorems}
\Crefname{theorem}{Theorem}{Theorems}
\crefname{proposition}{proposition}{propositions}
\Crefname{proposition}{Proposition}{Propositions}
\crefname{corollary}{corollary}{corollaries}
\Crefname{corollary}{Corollary}{Corollaries}
\crefname{lemma}{lemma}{lemmas}
\Crefname{lemma}{Lemma}{Lemmas}
\crefname{definition}{definition}{definitions}
\Crefname{definition}{Definition}{Definitions}
\crefname{example}{example}{examples}
\Crefname{example}{Example}{Examples}
\crefname{remark}{remark}{remarks}
\Crefname{remark}{Remark}{Remarks}
\crefname{equation}{equation}{equations}
\Crefname{equation}{Equation}{Equations}
\crefname{exercise}{exercise}{exercises}
\Crefname{exercise}{Exercise}{Exercises}

\title{Sylvester's Theorem and Reduced Linear Systems\\
\large From Three-Factor Products to Small Dynamical Models\\[0.6em]
\normalsize\textit{Lecture notes for an upper-level course in numerical analysis}}
\author{James M. Hyman\\
  \small Department of Mathematics, Tulane University, New Orleans, LA 70118\\
  \small \texttt{mhyman@tulane.edu}\qquad ORCID 0000-0001-5247-5794}
\date{July 29, 2026}

\begin{document}

\maketitle

\begin{abstract}
Large systems of linear equations often contain far fewer active degrees of freedom than their ambient dimension suggests. A particularly transparent instance occurs when a large matrix factors as
\[
  \mat{A}=\mat{T}\mat{S}\mat{W},
\]
where \(\mat{T}\in\F^{m\times n}\), \(\mat{S}\in\F^{n\times n}\), \(\mat{W}\in\F^{n\times m}\), and \(m\gg n\). The action of \(\mat{A}\) then passes through an \(n\)-dimensional intermediate space. Sylvester's theorem relates the nonzero eigenvalues of a product \(\mat{X}\mat{Y}\) to those of the reversed product \(\mat{Y}\mat{X}\). Applied cyclically, it shows that the nonzero spectrum of the large matrix \(\mat{A}\) is completely determined by either of the small matrices
\[
  \mat{B}=(\mat{W}\mat{T})\mat{S},
  \qquad
  \mat{C}=\mat{S}(\mat{W}\mat{T}).
\]
This article develops the theorem from first principles and explains the geometry behind the factorization. The result yields exact reduced systems, both for linear algebraic equations and for first-order linear ordinary differential equations. Worked examples carry this through in detail: spectral reduction, reconstruction of full-space eigenvectors, reduced solution of shifted systems, and exact integration of a three-dimensional ODE through a two-dimensional model. A separate section takes up singular values. Those of \(\mat{A}\), \(\mat{B}\), and \(\mat{C}\) do not agree in general, though they do when the bases are orthonormal, a case the same theorem settles once it is applied to \(\mat{A}^{*}\mat{A}\). The closing sections separate exact factorization from projection-based approximation, and flag what the theorem leaves undetermined: zero-eigenvalue structure, conditioning, and transient behavior. The material is classical. What this article adds is one continuous development, from first principles through to the limits of the theorem, pitched for an upper-level undergraduate course.
\end{abstract}



\tableofcontents

\section{Introduction: why a large system may be small}
\label{sec:introduction}

A matrix may be large because the state has many components, yet its interactions may depend on only a few combinations of those components. Suppose
\begin{equation}
  \mat{T}\in\F^{m\times n},
  \qquad
  \mat{S}\in\F^{n\times n},
  \qquad
  \mat{W}\in\F^{n\times m},
  \qquad
  m\gg n,
  \label{eq:dimensions}
\end{equation}
and define
\begin{equation}
  \mat{A}=\mat{T}\mat{S}\mat{W}\in\F^{m\times m}.
  \label{eq:A-factorization}
\end{equation}
The action of \(\mat{A}\) can be read from right to left:
\begin{equation}
  \F^m
  \xrightarrow{\;\mat{W}\;}
  \F^n
  \xrightarrow{\;\mat{S}\;}
  \F^n
  \xrightarrow{\;\mat{T}\;}
  \F^m.
  \label{eq:factorization-diagram}
\end{equation}
Each factor plays a distinct role: \(\mat{W}\) extracts or aggregates information from a full state, \(\mat{S}\) acts on that reduced information, and \(\mat{T}\) lifts the result back into the full space. Keeping these three roles separate is what makes the rest of the development readable. \Cref{fig:geometry} makes the consequence visible. Because every path between the two copies of \(\F^m\) runs through the narrow intermediate space, any question about \(\mat{A}\) that depends only on that passage can be settled at dimension \(n\). Identifying which questions those are is the work of the sections that follow.

\begin{figure}[!ht]
  \centering
  \begin{tikzpicture}[
    x=1mm,y=1mm,
    space/.style={draw,rounded corners=2pt,line width=0.4pt},
    sub/.style={draw,rounded corners=1.5pt,line width=0.3pt,fill=black!8},
    lbl/.style={font=\small},
    slbl/.style={font=\footnotesize},
    ar/.style={-{Stealth[length=2mm]},line width=0.5pt}
  ]
  \draw[space] (0,-13) rectangle (26,13);
  \node[slbl,anchor=north west] at (1.5,12) {$\F^m$};
  \draw[sub] (2,-11) rectangle (24,-4);
  \node[slbl] at (13,-7.5) {$\nullsp(\mat{W})$};
  \draw[space] (42,-5.5) rectangle (56,5.5);
  \node[slbl] at (49,0) {$\F^n$};
  \draw[space] (72,-5.5) rectangle (86,5.5);
  \node[slbl] at (79,0) {$\F^n$};
  \draw[space] (102,-13) rectangle (128,13);
  \node[slbl,anchor=north west] at (103.5,12) {$\F^m$};
  \draw[sub] (104,-11) rectangle (126,-4);
  \node[slbl] at (115,-7.5) {$\range(\mat{T})$};
  \draw[ar] (26,0) -- node[lbl,above] {$\mat{W}$} (42,0);
  \draw[ar] (56,0) -- node[lbl,above] {$\mat{S}$} (72,0);
  \draw[ar] (86,0) -- node[lbl,above] {$\mat{T}$} (102,0);
  \draw[ar,dashed] (13,13) -- (13,23) -- (115,23) -- (115,13.6);
  \node[lbl,above] at (64,23) {$\mat{A}=\mat{T}\mat{S}\mat{W}$};
  \end{tikzpicture}
  \caption{The factorization of \(\mat{A}\in\F^{m\times m}\) with \(m\gg n\); box size tracks dimension. Reading left to right, \(\mat{W}\) extracts, \(\mat{S}\) acts on the reduced state, and \(\mat{T}\) lifts. The dashed path is \(\mat{A}\) itself; the solid path is the route it takes. Every output lies in \(\range(\mat{T})\), and every direction in \(\nullsp(\mat{W})\) is annihilated, which is the content of \cref{eq:range-null-inclusions}. Because the solid path narrows to \(n\) dimensions, all of the nonzero spectral information in \(\mat{A}\) is already present in that small space.}
  \label{fig:geometry}
\end{figure}
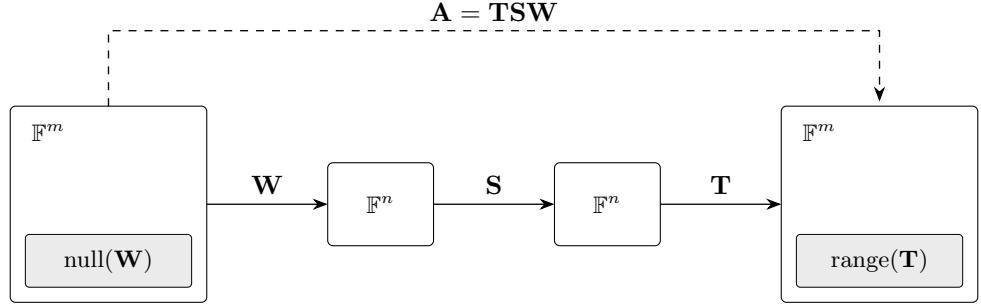

This factorization immediately suggests several questions.
\begin{enumerate}[label=\arabic*.]
  \item Which eigenvalues of the large matrix \(\mat{A}\) can be found from an \(n\times n\) matrix?
  \item Which reduced variables evolve without reference to all \(m\) components of the full state?
  \item When can a large linear system be solved by a small one?
  \item How does this exact structure relate to approximate projection-based model reduction?
\end{enumerate}

The central tool is a theorem usually associated with James Joseph Sylvester. In modern form, it states that \(\mat{X}\mat{Y}\) and \(\mat{Y}\mat{X}\) have the same nonzero eigenvalues, including algebraic multiplicities, even when the two products have different dimensions. Sylvester wrote about the ``latent roots'' of reversed matrix products in 1883 \parencite{sylvester1883equation}. The determinant identity used today is often called Sylvester's determinant theorem or Sylvester's determinant identity.

The computational motivation is equally important. A dense \(m\times m\) matrix requires \(m^2\) stored entries, and standard dense eigenvalue algorithms require on the order of \(m^3\) arithmetic operations. By contrast, the factors in \cref{eq:A-factorization} require about \(2mn+n^2\) entries. Forming \(\mat{W}\mat{T}\) costs on the order of \(mn^2\), after which the principal spectral work occurs in dimension \(n\). For \(m\gg n\), this is a decisive reduction. Similar considerations motivate the broad field of reduced-order modeling \parencite{antoulas2005approximation,benner2015survey,quarteroni2016reduced}.

The distinction between two uses of the word \emph{reduction} will be maintained throughout.

\begin{definition}[Exact factor reduction]
An exact factor reduction begins with a matrix that actually satisfies
\[
  \mat{A}=\mat{T}\mat{S}\mat{W}.
\]
All identities derived from this equality are exact.
\end{definition}

\begin{definition}[Projection-based approximation]
Given a general full matrix \(\mat{L}\in\F^{m\times m}\), a projection method chooses trial and test maps and constructs a smaller matrix \(\mat{L}_r\) so that
\[
  \mat{L}\approx \mat{T}\mat{L}_r\mat{W}.
\]
The equality between the lifted reduced operator and the original operator is generally approximate.
\end{definition}

Sylvester's theorem gives exact information about \(\mat{T}\mat{L}_r\mat{W}\) and \(\mat{L}_r\). It does not, by itself, guarantee that \(\mat{L}_r\) accurately approximates the spectrum or trajectories of \(\mat{L}\). That requires additional analysis of the chosen reduced space.

\subsection*{Intended readers and prerequisites}

This article is written for an upper-level undergraduate course in numerical analysis, and for readers with comparable background. The prerequisites are one course in linear algebra covering eigenvalues, characteristic polynomials, rank, and null spaces; familiarity with solving a scalar linear first-order differential equation; and enough contact with numerical computation to care about operation counts and conditioning. Schur complements are developed in \cref{sec:preliminaries} and are not assumed.

The mathematics is classical. Standard graduate references such as \textcite{horn2013matrix} and \textcite{golub2013matrix} state Sylvester's determinant identity and its spectral consequence, usually within a few lines. Collected here instead is the surrounding development at undergraduate pace: the geometry of the factorization, worked examples small enough to check by hand, the separation of exact factor reduction from projection-based approximation, and an account of what the theorem does not settle. Exercises with selected solutions appear in \cref{sec:exercises} and \cref{app:solutions}.

\subsection*{Learning objectives}

After studying this article, a reader should be able to:
\begin{enumerate}[label=\arabic*.]
  \item state and prove Sylvester's determinant identity;
  \item derive the characteristic-polynomial relationship for rectangular products;
  \item explain why \(\mat{A}\), \(\mat{B}\), and \(\mat{C}\) have the same nonzero eigenvalues;
  \item interpret \(\mat{T}\) as a lifting map and \(\mat{W}\) as an extraction or testing map;
  \item construct exact reduced systems for algebraic equations and linear ODEs;
  \item reconstruct full-space eigenvectors and trajectories from reduced variables;
  \item distinguish exact factor reduction from approximate projection-based reduction; and
  \item identify when spectral reduction is insufficient, and explain what the theorem does not determine.
\end{enumerate}

The development runs in a single line. \Cref{sec:preliminaries} collects the rank, spectral, and Schur complement facts the proof requires. \Cref{sec:sylvester} proves Sylvester's identity for a product of two rectangular matrices. It then draws out what the identity says about characteristic polynomials, about eigenvectors, and about the eigenvalue zero, which behaves differently from the rest. \Cref{sec:three-factors} extends the result to the three-factor form and introduces the two small matrices that carry the nonzero spectrum. \Cref{sec:worked-linear-algebra} works an example by hand before any application appears, and closes by describing where factorizations of this shape come from.

\Cref{sec:singular-values} then establishes what the theorem does not give. Its position is deliberate. The projection methods recommended in \Cref{sec:projection} are built from singular values, and singular values do not transfer.

\Cref{sec:algebraic-systems} applies the exact theory to linear systems, powers, resolvents, and matrix functions. \Cref{sec:ode-reduction} and \Cref{sec:worked-ode} do the same for first-order linear differential equations, and \Cref{sec:shifted-stability} treats forcing and asymptotic stability. \Cref{sec:projection} turns from exact reduction to approximate reduction. It marks the boundary between what the algebra settles and what a modeling choice must justify. \Cref{sec:computational-workflow} treats computation and cost, and \Cref{sec:limitations} collects the cautions. Exercises, a conclusion, selected solutions, and a compact reference sheet close the article.

\section{Preliminaries}
\label{sec:preliminaries}

We work over \(\F=\R\) or \(\F=\C\). Unless otherwise stated, matrices need not be symmetric, normal, or diagonalizable.

\subsection{Range, null space, and rank}

For a matrix \(\mat{M}\), write \(\range(\mat{M})\) for its column space and \(\nullsp(\mat{M})\) for its null space. The basic rank inequalities
\begin{equation}
  \rank(\mat{X}\mat{Y})
  \leq
  \min\{\rank(\mat{X}),\rank(\mat{Y})\}
  \label{eq:rank-product}
\end{equation}
imply that the product of a tall matrix and a wide matrix can have rank no larger than the intermediate dimension.

Applied to \cref{eq:A-factorization},
\begin{equation}
  \rank(\mat{A})\leq n.
  \label{eq:rank-A}
\end{equation}
The rank bound forces at least \(m-n\) zero eigenvalues, counted with algebraic multiplicity.

Two geometric inclusions are equally immediate:
\begin{equation}
  \range(\mat{A})\subseteq \range(\mat{T}),
  \qquad
  \nullsp(\mat{W})\subseteq \nullsp(\mat{A}).
  \label{eq:range-null-inclusions}
\end{equation}
Every output of \(\mat{A}\) lies in the lifted subspace \(\range(\mat{T})\), while every direction invisible to \(\mat{W}\) is annihilated by \(\mat{A}\).

\subsection{Eigenvalues and characteristic polynomials}

For a square matrix \(\mat{M}\in\F^{k\times k}\), an eigenvalue \(\lambda\in\C\) satisfies
\[
  \mat{M}\vect{v}=\lambda\vect{v}
\]
for some nonzero vector \(\vect{v}\). The eigenvalues are the roots of the characteristic polynomial
\begin{equation}
  p_{\mat{M}}(\lambda)
  =\det(\lambda\I_k-\mat{M}).
  \label{eq:characteristic-polynomial}
\end{equation}
In practice one almost never forms a characteristic polynomial to compute eigenvalues. Its coefficients can be extremely sensitive to perturbation, so recovering accurate roots from them is unreliable even in modest dimensions. Determinant identities remain valuable as proof tools. That is how they are used here: to establish spectral equivalence, not to compute anything.

\subsection{Schur complements}

The proof of Sylvester's determinant identity uses a standard block determinant formula. If \(\mat{P}\) is invertible, then
\begin{equation}
  \det
  \begin{pmatrix}
    \mat{P}&\mat{Q}\\
    \mat{R}&\mat{S}
  \end{pmatrix}
  =
  \det(\mat{P})
  \det(\mat{S}-\mat{R}\mat{P}^{-1}\mat{Q}).
  \label{eq:schur-P}
\end{equation}
If \(\mat{S}\) is invertible, then
\begin{equation}
  \det
  \begin{pmatrix}
    \mat{P}&\mat{Q}\\
    \mat{R}&\mat{S}
  \end{pmatrix}
  =
  \det(\mat{S})
  \det(\mat{P}-\mat{Q}\mat{S}^{-1}\mat{R}).
  \label{eq:schur-S}
\end{equation}
Both identities follow from block Gaussian elimination. Factor the block matrix into block triangular factors, then use the fact that a block triangular determinant is the product of its diagonal blocks. Having two formulas for the same determinant is the entire mechanism of the proof in \cref{sec:sylvester}. Standard references for matrix analysis and computation include \textcite{horn2013matrix} and \textcite{golub2013matrix}.

\section{Sylvester's theorem}
\label{sec:sylvester}

\subsection{The determinant identity}

\begin{theorem}[Sylvester's determinant identity]
\label{thm:sylvester-determinant}
Let
\[
  \mat{X}\in\F^{m\times n},
  \qquad
  \mat{Y}\in\F^{n\times m}.
\]
Then
\begin{equation}
  \det(\I_m+\mat{X}\mat{Y})
  =
  \det(\I_n+\mat{Y}\mat{X}).
  \label{eq:sylvester-determinant}
\end{equation}
\end{theorem}

\begin{proof}
Consider the block matrix
\begin{equation}
  \mat{K}
  =
  \begin{pmatrix}
    \I_m&\mat{X}\\
    -\mat{Y}&\I_n
  \end{pmatrix}.
  \label{eq:block-K}
\end{equation}
Using the upper-left block \(\I_m\) in the Schur-complement formula gives
\[
  \det(\mat{K})
  =\det(\I_m)
   \det(\I_n+\mat{Y}\mat{X})
  =\det(\I_n+\mat{Y}\mat{X}).
\]
Using the lower-right block \(\I_n\) instead gives
\[
  \det(\mat{K})
  =\det(\I_n)
   \det(\I_m+\mat{X}\mat{Y})
  =\det(\I_m+\mat{X}\mat{Y}).
\]
The two expressions are equal because they are both \(\det(\mat{K})\).
\end{proof}

The theorem is striking because the two determinants can have different sizes. One is \(m\times m\), the other \(n\times n\), yet they agree exactly. \Cref{cor:rectangular-characteristic} will explain why. The extra \(m-n\) eigenvalues of \(\mat{X}\mat{Y}\) are all zero, so the matching eigenvalues of \(\I_m+\mat{X}\mat{Y}\) equal one and contribute nothing to the determinant.

\subsection{The characteristic-polynomial form}

\begin{corollary}[Rectangular product theorem]
\label{cor:rectangular-characteristic}
Suppose \(m\geq n\). Then
\begin{equation}
  \det(\lambda\I_m-\mat{X}\mat{Y})
  =
  \lambda^{m-n}
  \det(\lambda\I_n-\mat{Y}\mat{X}).
  \label{eq:rectangular-characteristic}
\end{equation}
Consequently, \(\mat{X}\mat{Y}\) and \(\mat{Y}\mat{X}\) have the same nonzero eigenvalues, including algebraic multiplicities.
\end{corollary}

\begin{proof}
For \(\lambda\neq0\), factor \(\lambda\) from each determinant:
\begin{align*}
  \det(\lambda\I_m-\mat{X}\mat{Y})
  &=\lambda^m
    \det\!\left(\I_m-\lambda^{-1}\mat{X}\mat{Y}\right)\\
  &=\lambda^m
    \det\!\left(\I_n-\lambda^{-1}\mat{Y}\mat{X}\right)\\
  &=\lambda^{m-n}
    \det(\lambda\I_n-\mat{Y}\mat{X}),
\end{align*}
where \cref{thm:sylvester-determinant} is applied with
\(\mat{X}\) replaced by \(-\lambda^{-1}\mat{X}\). Both sides of
\cref{eq:rectangular-characteristic} are polynomials in \(\lambda\), so equality for all nonzero \(\lambda\) implies equality for every \(\lambda\).
\end{proof}

\begin{remark}[Terminology and history]
Sylvester's 1883 paper used the language of latent roots and emphasized that reversing the order of two factors preserves the nonzero spectral information \parencite{sylvester1883equation}. Modern texts often state the result through \cref{eq:sylvester-determinant} or \cref{eq:rectangular-characteristic}. Several other theorems also bear Sylvester's name, so the phrase \emph{Sylvester's determinant identity} is the least ambiguous in this setting.
\end{remark}

\subsection{The eigenvector mapping}

The determinant proof establishes algebraic multiplicities. A direct eigenvector proof explains the geometry.

\begin{proposition}[Eigenvector transfer]
\label{prop:eigenvector-transfer}
Let \(\lambda\neq0\).
\begin{enumerate}[label=(\alph*)]
  \item If \(\mat{Y}\mat{X}\vect{v}=\lambda\vect{v}\), then
  \(\mat{X}\vect{v}\neq\vect{0}\) and
  \[
    \mat{X}\mat{Y}(\mat{X}\vect{v})
    =\lambda(\mat{X}\vect{v}).
  \]
  \item If \(\mat{X}\mat{Y}\vect{u}=\lambda\vect{u}\), then
  \(\mat{Y}\vect{u}\neq\vect{0}\) and
  \[
    \mat{Y}\mat{X}(\mat{Y}\vect{u})
    =\lambda(\mat{Y}\vect{u}).
  \]
\end{enumerate}
\end{proposition}

\begin{proof}
Suppose \(\mat{Y}\mat{X}\vect{v}=\lambda\vect{v}\). If
\(\mat{X}\vect{v}=\vect{0}\), then the left side is zero, which would imply \(\lambda\vect{v}=\vect{0}\). Since \(\lambda\neq0\), this contradicts \(\vect{v}\neq\vect{0}\), so \(\mat{X}\vect{v}\neq\vect{0}\), and
\[
  \mat{X}\mat{Y}(\mat{X}\vect{v})
  =\mat{X}(\mat{Y}\mat{X}\vect{v})
  =\lambda\mat{X}\vect{v}.
\]
The reverse statement is analogous.
\end{proof}

\subsection{Why zero is different}

The theorem does not say that the two products have identical behavior at zero. If \(m>n\), the larger product has at least \(m-n\) additional zero eigenvalues. Even when both products are square of the same size, their Jordan blocks at zero can differ.

\begin{example}[Different null structures]
Let
\[
  \mat{X}=
  \begin{pmatrix}
    1&0\\
    0&0
  \end{pmatrix},
  \qquad
  \mat{Y}=
  \begin{pmatrix}
    0&1\\
    0&0
  \end{pmatrix}.
\]
Then
\[
  \mat{X}\mat{Y}
  =
  \begin{pmatrix}
    0&1\\
    0&0
  \end{pmatrix},
  \qquad
  \mat{Y}\mat{X}
  =
  \begin{pmatrix}
    0&0\\
    0&0
  \end{pmatrix}.
\]
Both matrices have only the eigenvalue zero, but the first is a nonzero nilpotent matrix with a Jordan block of size two, whereas the second is the zero matrix.
\end{example}

The detailed relation between the elementary divisors of \(\mat{X}\mat{Y}\) and \(\mat{Y}\mat{X}\), including their zero Jordan blocks, was analyzed by \textcite{flanders1951elementary}. For reduced-system analysis, the most frequently used part is the exact agreement of the nonzero spectrum.

\section{Three factors and two reduced matrices}
\label{sec:three-factors}

Return to
\[
  \mat{A}=\mat{T}\mat{S}\mat{W}.
\]
Define
\begin{equation}
  \mat{M}=\mat{W}\mat{T}\in\F^{n\times n},
  \label{eq:M-definition}
\end{equation}
and then
\begin{equation}
  \mat{B}=\mat{M}\mat{S}=(\mat{W}\mat{T})\mat{S},
  \qquad
  \mat{C}=\mat{S}\mat{M}=\mat{S}(\mat{W}\mat{T}).
  \label{eq:B-C-definition}
\end{equation}
Two reduced matrices appear because a product of three factors can be grouped in two ways. Writing \(\mat{A}=(\mat{T}\mat{S})\mat{W}\) and applying \cref{cor:rectangular-characteristic} with \(\mat{X}=\mat{T}\mat{S}\) and \(\mat{Y}=\mat{W}\) produces \(\mat{B}\). Writing \(\mat{A}=\mat{T}(\mat{S}\mat{W})\) and applying the same corollary with \(\mat{X}=\mat{T}\) and \(\mat{Y}=\mat{S}\mat{W}\) produces \(\mat{C}\). The three-factor result needs no new theorem, only the two-factor one used twice.

The dimensions and interpretations are summarized in \cref{tab:dimensions}.

\begin{table}[htbp]
  \centering
  \caption{Dimensions and roles of the matrices in the factorization.}
  \label{tab:dimensions}
  \begin{tabular}{lll}
    \toprule
    Matrix & Dimensions & Typical role \\
    \midrule
    \(\mat{T}\) & \(m\times n\) & lifting or synthesis \\
    \(\mat{S}\) & \(n\times n\) & reduced interaction \\
    \(\mat{W}\) & \(n\times m\) & extraction, aggregation, or testing \\
    \(\mat{A}\) & \(m\times m\) & full-space operator \\
    \(\mat{B}\) & \(n\times n\) & observable-space operator \\
    \(\mat{C}\) & \(n\times n\) & lifted-coordinate operator \\
    \bottomrule
  \end{tabular}
\end{table}

\subsection{The main spectral reduction theorem}

\begin{theorem}[Three-factor spectral reduction]
\label{thm:three-factor}
Let the matrices have the dimensions in \cref{eq:dimensions}, with \(m\geq n\). Then
\begin{equation}
  \det(\lambda\I_m-\mat{A})
  =
  \lambda^{m-n}\det(\lambda\I_n-\mat{B})
  =
  \lambda^{m-n}\det(\lambda\I_n-\mat{C}).
  \label{eq:three-factor-characteristic}
\end{equation}
Therefore \(\mat{A}\), \(\mat{B}\), and \(\mat{C}\) have the same nonzero eigenvalues, including algebraic multiplicities.
\end{theorem}

\begin{proof}
First set
\[
  \mat{X}=\mat{T}\mat{S},
  \qquad
  \mat{Y}=\mat{W}.
\]
Then \(\mat{X}\mat{Y}=\mat{A}\) and
\(\mat{Y}\mat{X}=\mat{B}\). By \cref{cor:rectangular-characteristic},
\[
  \det(\lambda\I_m-\mat{A})
  =\lambda^{m-n}\det(\lambda\I_n-\mat{B}).
\]
Next set
\[
  \mat{X}=\mat{T},
  \qquad
  \mat{Y}=\mat{S}\mat{W}.
\]
Then \(\mat{X}\mat{Y}=\mat{A}\) and
\(\mat{Y}\mat{X}=\mat{C}\), which gives the second equality.
\end{proof}

\begin{corollary}[The reduced matrices]
\label{cor:B-C-spectrum}
The matrices \(\mat{B}=\mat{M}\mat{S}\) and
\(\mat{C}=\mat{S}\mat{M}\) have the same characteristic polynomial:
\begin{equation}
  \det(\lambda\I_n-\mat{B})
  =
  \det(\lambda\I_n-\mat{C}).
  \label{eq:B-C-characteristic}
\end{equation}
\end{corollary}

\begin{proof}
Apply \cref{cor:rectangular-characteristic} with the two square matrices
\(\mat{M}\) and \(\mat{S}\). Because their dimensions are equal, no additional power of \(\lambda\) appears.
\end{proof}

\begin{proposition}[When \(\mat{B}\) and \(\mat{C}\) are similar]
\label{prop:B-C-similar}
If either \(\mat{S}\) or \(\mat{M}=\mat{W}\mat{T}\) is invertible, then \(\mat{B}\) and \(\mat{C}\) are similar.
\end{proposition}

\begin{proof}
If \(\mat{S}\) is invertible, then
\[
  \mat{C}
  =\mat{S}\mat{M}
  =\mat{S}(\mat{M}\mat{S})\mat{S}^{-1}
  =\mat{S}\mat{B}\mat{S}^{-1}.
\]
If \(\mat{M}\) is invertible, then
\[
  \mat{C}
  =\mat{S}\mat{M}
  =\mat{M}^{-1}(\mat{M}\mat{S})\mat{M}
  =\mat{M}^{-1}\mat{B}\mat{M}.
\]
\end{proof}

Similarity is stronger than equality of characteristic polynomials. Without an invertibility assumption, \(\mat{B}\) and \(\mat{C}\) still have the same characteristic polynomial, but they need not be similar at the zero eigenvalue.

\subsection{Intertwining identities}

The most useful reduced-system relations are not determinant identities but simple associativity identities:
\begin{equation}
  \mat{A}\mat{T}
  =\mat{T}\mat{C},
  \qquad
  \mat{W}\mat{A}
  =\mat{B}\mat{W}.
  \label{eq:intertwining}
\end{equation}
Both follow by regrouping the three factors:
\[
  \mat{A}\mat{T}
  =\mat{T}\mat{S}\mat{W}\mat{T}
  =\mat{T}\mat{C},
\]
and
\[
  \mat{W}\mat{A}
  =\mat{W}\mat{T}\mat{S}\mat{W}
  =\mat{B}\mat{W}.
\]
These equations say that \(\mat{T}\) and \(\mat{W}\) connect the full and reduced operators in a consistent way. \Cref{fig:intertwining} draws them as squares. That reading is worth having because it makes the identities usable without algebra: any statement obtained along one side of a square transports to the other, which is how the eigenvector and trajectory results below are obtained.

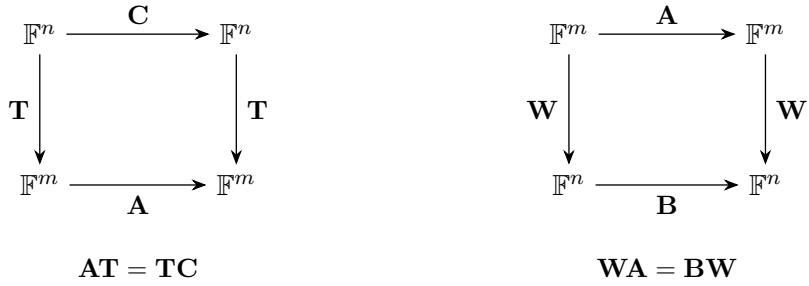
\begin{figure}[!ht]
  \centering
  \begin{tikzpicture}[
    x=1mm,y=1mm,
    ar/.style={-{Stealth[length=2mm]},line width=0.5pt},
    lbl/.style={font=\small}
  ]
  \node (a1) at (0,0)   {$\F^n$};
  \node (a2) at (26,0)  {$\F^n$};
  \node (a3) at (0,-20) {$\F^m$};
  \node (a4) at (26,-20){$\F^m$};
  \draw[ar] (a1) -- node[lbl,above] {$\mat{C}$} (a2);
  \draw[ar] (a1) -- node[lbl,left]  {$\mat{T}$} (a3);
  \draw[ar] (a2) -- node[lbl,right] {$\mat{T}$} (a4);
  \draw[ar] (a3) -- node[lbl,below] {$\mat{A}$} (a4);
  \node[lbl] at (13,-31) {$\mat{A}\mat{T}=\mat{T}\mat{C}$};
  \begin{scope}[xshift=70mm]
  \node (b1) at (0,0)   {$\F^m$};
  \node (b2) at (26,0)  {$\F^m$};
  \node (b3) at (0,-20) {$\F^n$};
  \node (b4) at (26,-20){$\F^n$};
  \draw[ar] (b1) -- node[lbl,above] {$\mat{A}$} (b2);
  \draw[ar] (b1) -- node[lbl,left]  {$\mat{W}$} (b3);
  \draw[ar] (b2) -- node[lbl,right] {$\mat{W}$} (b4);
  \draw[ar] (b3) -- node[lbl,below] {$\mat{B}$} (b4);
  \node[lbl] at (13,-31) {$\mat{W}\mat{A}=\mat{B}\mat{W}$};
  \end{scope}
  \end{tikzpicture}
  \caption{The intertwining identities as commuting squares, where \(\mat{T}\) lifts from \(\F^n\) to \(\F^m\) and \(\mat{W}\) extracts in the other direction. Each square commutes: applying the reduced operator and then crossing gives the same map as crossing and then applying \(\mat{A}\). These two squares are the mechanism behind the eigenvector transfer of \cref{prop:reduced-eigenvectors}, the invariance of \(\range(\mat{T})\) in \cref{cor:invariant-range}, and both reduced systems of \cref{sec:ode-reduction}.}
  \label{fig:intertwining}
\end{figure}

\begin{corollary}[Invariant lifted subspace]
\label{cor:invariant-range}
The subspace \(\range(\mat{T})\) is invariant under \(\mat{A}\).
\end{corollary}

\begin{proof}
For every \(\vect{y}\in\F^n\),
\[
  \mat{A}(\mat{T}\vect{y})
  =\mat{T}(\mat{C}\vect{y})
  \in\range(\mat{T}).
\]
\end{proof}

\begin{proposition}[Reduced eigenvectors]
\label{prop:reduced-eigenvectors}
Let \(\lambda\neq0\).
\begin{enumerate}[label=(\alph*)]
  \item If \(\mat{C}\vect{v}=\lambda\vect{v}\), then
  \(\mat{T}\vect{v}\neq\vect{0}\) and
  \[
    \mat{A}(\mat{T}\vect{v})
    =\lambda\mat{T}\vect{v}.
  \]
  \item If \(\mat{B}\vect{u}=\lambda\vect{u}\), then
  \(\mat{T}\mat{S}\vect{u}\neq\vect{0}\) and
  \[
    \mat{A}(\mat{T}\mat{S}\vect{u})
    =\lambda\mat{T}\mat{S}\vect{u}.
  \]
\end{enumerate}
\end{proposition}

\begin{proof}
For part (a), use \(\mat{A}\mat{T}=\mat{T}\mat{C}\):
\[
  \mat{A}(\mat{T}\vect{v})
  =\mat{T}\mat{C}\vect{v}
  =\lambda\mat{T}\vect{v}.
\]
If \(\mat{T}\vect{v}=\vect{0}\), then
\(\mat{C}\vect{v}=\mat{S}\mat{W}\mat{T}\vect{v}=\vect{0}\), contradicting \(\lambda\neq0\) and \(\vect{v}\neq\vect{0}\). Part (b) follows from
\[
  \mat{A}\mat{T}\mat{S}
  =\mat{T}\mat{S}\mat{W}\mat{T}\mat{S}
  =\mat{T}\mat{S}\mat{B}.
\]
The nonvanishing argument is the same.
\end{proof}

\subsection{The biorthogonal case}

A particularly useful choice satisfies
\begin{equation}
  \mat{W}\mat{T}=\I_n.
  \label{eq:biorthogonal}
\end{equation}
Then
\begin{equation}
  \mat{B}=\mat{C}=\mat{S}.
  \label{eq:B-C-equal-S}
\end{equation}
The composite map
\begin{equation}
  \mat{P}=\mat{T}\mat{W}
  \label{eq:P-definition}
\end{equation}
is a projection because
\[
  \mat{P}^2
  =\mat{T}(\mat{W}\mat{T})\mat{W}
  =\mat{T}\mat{W}
  =\mat{P}.
\]
Its range is \(\range(\mat{T})\), and its null space is \(\nullsp(\mat{W})\). Every full vector splits as
\begin{equation}
  \vect{x}
  =\mat{P}\vect{x}+(\I_m-\mat{P})\vect{x},
  \label{eq:projection-decomposition}
\end{equation}
where the first term belongs to the active lifted subspace and the second is invisible to \(\mat{W}\).

When the columns of \(\mat{T}\) are orthonormal, the standard choice
\begin{equation}
  \mat{W}=\mat{T}^{*}
  \label{eq:orthogonal-W}
\end{equation}
gives \(\mat{W}\mat{T}=\I_n\), and \(\mat{P}=\mat{T}\mat{T}^{*}\) is the orthogonal projector onto \(\range(\mat{T})\).

\section{A worked linear algebra example}
\label{sec:worked-linear-algebra}

The following example is small enough to compute by hand but large enough to show that \(\mat{B}\) and \(\mat{C}\) need not be equal.

\begin{example}[A three-dimensional operator reduced to dimension two]
\label{ex:three-by-three}
Let
\begin{equation}
  \mat{T}=
  \begin{pmatrix}
    1&0\\
    0&1\\
    1&1
  \end{pmatrix},
  \qquad
  \mat{S}=
  \begin{pmatrix}
    -2&0\\
    0&-3
  \end{pmatrix},
  \qquad
  \mat{W}=
  \begin{pmatrix}
    1&0&0\\
    0&1&1
  \end{pmatrix}.
  \label{eq:example1-factors}
\end{equation}
Then
\begin{equation}
  \mat{M}=\mat{W}\mat{T}
  =
  \begin{pmatrix}
    1&0\\
    1&2
  \end{pmatrix}.
  \label{eq:example1-M}
\end{equation}
The two reduced matrices are
\begin{align}
  \mat{B}=\mat{M}\mat{S}
  &=
  \begin{pmatrix}
    -2&0\\
    -2&-6
  \end{pmatrix},
  \label{eq:example1-B}\\
  \mat{C}=\mat{S}\mat{M}
  &=
  \begin{pmatrix}
    -2&0\\
    -3&-6
  \end{pmatrix}.
  \label{eq:example1-C}
\end{align}
They are different matrices, but both are triangular with eigenvalues \(-2\) and \(-6\).

The full matrix is
\begin{equation}
  \mat{A}=\mat{T}\mat{S}\mat{W}
  =
  \begin{pmatrix}
    -2&0&0\\
    0&-3&-3\\
    -2&-3&-3
  \end{pmatrix}.
  \label{eq:example1-A}
\end{equation}
A direct determinant calculation gives
\begin{equation}
  \det(\lambda\I_3-\mat{A})
  =\lambda(\lambda+2)(\lambda+6).
  \label{eq:example1-charA}
\end{equation}
Meanwhile,
\begin{equation}
  \det(\lambda\I_2-\mat{B})
  =
  \det(\lambda\I_2-\mat{C})
  =(\lambda+2)(\lambda+6).
  \label{eq:example1-charBC}
\end{equation}
\Cref{eq:three-factor-characteristic} becomes
\[
  \lambda(\lambda+2)(\lambda+6)
  =\lambda(\lambda+2)(\lambda+6).
\]
The extra factor \(\lambda\) is the inactive dimension in the full space.
\end{example}

\subsection{Reconstructing full eigenvectors}

For \(\lambda=-6\), \cref{eq:example1-C} gives
\[
  \mat{C}
  \begin{pmatrix}0\\1\end{pmatrix}
  =-6
  \begin{pmatrix}0\\1\end{pmatrix}.
\]
By \cref{prop:reduced-eigenvectors}, lift with \(\mat{T}\):
\begin{equation}
  \vect{x}_{-6}
  =\mat{T}
  \begin{pmatrix}0\\1\end{pmatrix}
  =
  \begin{pmatrix}0\\1\\1\end{pmatrix}.
  \label{eq:example1-eigenvector-minus6}
\end{equation}
Multiplying by \(\mat{A}\) confirms this:
\[
  \mat{A}
  \begin{pmatrix}0\\1\\1\end{pmatrix}
  =
  \begin{pmatrix}0\\-6\\-6\end{pmatrix}
  =-6
  \begin{pmatrix}0\\1\\1\end{pmatrix}.
\]

For \(\lambda=-2\), an eigenvector of \(\mat{C}\) is
\[
  \vect{v}_{-2}
  =
  \begin{pmatrix}4\\-3\end{pmatrix},
\]
which lifts to
\begin{equation}
  \vect{x}_{-2}
  =\mat{T}\vect{v}_{-2}
  =
  \begin{pmatrix}4\\-3\\1\end{pmatrix}.
  \label{eq:example1-eigenvector-minus2}
\end{equation}
Direct multiplication verifies \(\mat{A}\vect{x}_{-2}=-2\vect{x}_{-2}\).

The zero eigenvalue requires separate treatment. Solving \(\mat{A}\vect{x}=\vect{0}\) gives, for example,
\begin{equation}
  \vect{x}_{0}
  =
  \begin{pmatrix}0\\-1\\1\end{pmatrix}.
  \label{eq:example1-zero-vector}
\end{equation}
This vector belongs to \(\nullsp(\mat{W})\), as predicted by \cref{eq:range-null-inclusions}.

\subsection{What the example teaches}

The example separates three ideas that are sometimes conflated.
\begin{enumerate}[label=\arabic*.]
  \item The matrices \(\mat{B}\) and \(\mat{C}\) can differ entry by entry.
  \item Their characteristic polynomials are nevertheless identical.
  \item Beyond what the reduced matrices already contain, the large matrix \(\mat{A}\) contributes only additional zero eigenvalues.
\end{enumerate}

For an \(m\times m\) problem with \(m\) in the millions and \(n\) in the tens or hundreds, the same algebra replaces an infeasible full spectral calculation by a manageable reduced one.

Factorizations of this shape arise in two ways. They can be built deliberately, by projecting a large operator onto a chosen low-dimensional subspace, which is the subject of \cref{sec:projection}. They also arise structurally, whenever the components of a large state interact only through a small number of aggregated quantities, so that the coupling passes through \(\mat{W}\) before returning through \(\mat{T}\). The shifted systems of \cref{sec:shifted-stability}, where a low-rank interaction sits on top of uniform decay, are of the second kind.

\section{Singular values of the three matrices}
\label{sec:singular-values}

Sylvester's theorem settles the nonzero eigenvalues of \(\mat{A}\), \(\mat{B}\), and \(\mat{C}\). It says nothing about their singular values. That distinction matters as soon as norms, conditioning, or sensitivity to perturbation enter the discussion, which in numerical work happens almost immediately. For the singular value decomposition itself, see \textcite{trefethen1997numerical} or \textcite{golub2013matrix}.

\subsection{The three matrices have different singular values}

The example of \cref{sec:worked-linear-algebra} already shows the gap. Its two reduced matrices share the eigenvalues \(-2\) and \(-6\), and yet they differ in every singular value, as recorded in \cref{tab:singular-values}.

\begin{table}[htbp]
  \centering
  \caption{Eigenvalues and singular values for the matrices of \cref{ex:three-by-three}. The nonzero eigenvalues agree, as Sylvester's theorem requires. The singular values do not.}
  \label{tab:singular-values}
  \begin{tabular}{llll}
    \toprule
    Matrix & Nonzero eigenvalues & Singular values & Condition number \\
    \midrule
    \(\mat{A}\) & \(-2,\ -6\) & \(6.194,\ 2.373,\ 0\) & \(\infty\) \\
    \(\mat{B}\) & \(-2,\ -6\) & \(6.359,\ 1.887\) & \(3.370\) \\
    \(\mat{C}\) & \(-2,\ -6\) & \(6.772,\ 1.772\) & \(3.822\) \\
    \bottomrule
  \end{tabular}
\end{table}

One feature of the table is worth pausing over. The two products agree. Both \(\mat{B}\) and \(\mat{C}\) have determinant \(12\), which their common characteristic polynomial forces, so in each case the singular values multiply to \(12\). What differs is how that product is distributed between the two, and a condition number measures exactly that distribution.

No Sylvester-type identity is available here. The singular values of \(\mat{A}\) are the positive square roots of the eigenvalues of
\begin{equation}
  \mat{A}^{*}\mat{A}
  =\mat{W}^{*}\mat{S}^{*}(\mat{T}^{*}\mat{T})\mat{S}\mat{W},
  \label{eq:AstarA}
\end{equation}
and the inner factor \(\mat{T}^{*}\mat{T}\) does not cancel. The outcome depends on \(\mat{T}\) and \(\mat{W}\) separately, not only on the product \(\mat{M}=\mat{W}\mat{T}\) that determines \(\mat{B}\) and \(\mat{C}\).

Comparing shapes makes the situation clearer. \Cref{fig:shapes} places the factorization beside the thin singular value decomposition of the same matrix. The two have the same shape, which rules out any dimensional obstruction to a singular-value analogue of \cref{thm:three-factor}. What blocks the analogue is the loss of orthonormality, and the rest of this section works out what follows from that.

\begin{figure}[!ht]
  \centering
  \begin{tikzpicture}[
    x=1mm,y=1mm,
    bx/.style={draw,line width=0.4pt},
    dim/.style={font=\scriptsize},
    nm/.style={font=\small},
    op/.style={font=\small}
  ]
  \def\M{22}
  \def\N{6}
  \node[nm,anchor=east] at (-4,-11) {$\mat{A}=\mat{T}\mat{S}\mat{W}$};
  \draw[bx] (0,0) rectangle (\N,-\M);          \node[nm] at (\N/2,-\M/2) {$\mat{T}$};
  \node[dim] at (\N/2,-\M-4) {$m\times n$};
  \node[op] at (\N+4,-11) {$\times$};
  \draw[bx] (\N+8,-8) rectangle (\N+8+\N,-8-\N);  \node[nm] at (\N+8+\N/2,-8-\N/2) {$\mat{S}$};
  \node[dim] at (\N+8+\N/2,-\M-4) {$n\times n$};
  \node[op] at (2*\N+12,-11) {$\times$};
  \draw[bx] (2*\N+16,-8) rectangle (2*\N+16+\M,-8-\N); \node[nm] at (2*\N+16+\M/2,-8-\N/2) {$\mat{W}$};
  \node[dim] at (2*\N+16+\M/2,-\M-4) {$n\times m$};
  \node[op] at (2*\N+\M+20,-11) {$=$};
  \draw[bx] (2*\N+\M+24,0) rectangle (2*\N+\M+24+\M,-\M); \node[nm] at (2*\N+\M+24+\M/2,-\M/2) {$\mat{A}$};
  \node[dim] at (2*\N+\M+24+\M/2,-\M-4) {$m\times m$};
  \begin{scope}[yshift=-42mm]
  \node[nm,anchor=east] at (-4,-11) {$\mat{M}=\mat{W}\mat{T}$};
  \draw[bx] (0,-8) rectangle (\M,-8-\N);   \node[nm] at (\M/2,-8-\N/2) {$\mat{W}$};
  \node[dim] at (\M/2,-\M-4) {$n\times m$};
  \node[op] at (\M+4,-11) {$\times$};
  \draw[bx] (\M+8,0) rectangle (\M+8+\N,-\M);  \node[nm] at (\M+8+\N/2,-\M/2) {$\mat{T}$};
  \node[dim] at (\M+8+\N/2,-\M-4) {$m\times n$};
  \node[op] at (\M+\N+12,-11) {$=$};
  \draw[bx] (\M+\N+16,-8) rectangle (\M+\N+16+\N,-8-\N); \node[nm] at (\M+\N+16+\N/2,-8-\N/2) {$\mat{M}$};
  \node[dim] at (\M+\N+16+\N/2,-\M-4) {$n\times n$};
  \end{scope}
  \begin{scope}[yshift=-84mm]
  \node[nm,anchor=east] at (-4,-11) {$\mat{A}=\mat{U}\mat{\Sigma}\mat{V}^{*}$};
  \draw[bx] (0,0) rectangle (\N,-\M);          \node[nm] at (\N/2,-\M/2) {$\mat{U}$};
  \node[dim] at (\N/2,-\M-4) {$m\times n$};
  \node[op] at (\N+4,-11) {$\times$};
  \draw[bx] (\N+8,-8) rectangle (\N+8+\N,-8-\N);  \node[nm] at (\N+8+\N/2,-8-\N/2) {$\mat{\Sigma}$};
  \node[dim] at (\N+8+\N/2,-\M-4) {$n\times n$};
  \node[op] at (2*\N+12,-11) {$\times$};
  \draw[bx] (2*\N+16,-8) rectangle (2*\N+16+\M,-8-\N); \node[nm] at (2*\N+16+\M/2,-8-\N/2) {$\mat{V}^{*}$};
  \node[dim] at (2*\N+16+\M/2,-\M-4) {$n\times m$};
  \node[op] at (2*\N+\M+20,-11) {$=$};
  \draw[bx] (2*\N+\M+24,0) rectangle (2*\N+\M+24+\M,-\M); \node[nm] at (2*\N+\M+24+\M/2,-\M/2) {$\mat{A}$};
  \node[dim] at (2*\N+\M+24+\M/2,-\M-4) {$m\times m$};
  \end{scope}
  \end{tikzpicture}
  \caption{Shape arithmetic for \(\mat{A}\in\F^{m\times m}\) with \(m\gg n\), drawn to scale. The factorization passes through an \(n\)-dimensional waist, so \(\mat{A}\) has rank at most \(n\). Reversing the order collapses \(\mat{W}\mat{T}\) to \(n\times n\), which is the step Sylvester's theorem exploits, and both \(\mat{B}=\mat{M}\mat{S}\) and \(\mat{C}=\mat{S}\mat{M}\) inherit that size. The thin singular value decomposition in the last row has the same shape signature as the first. What separates the two is not size but constraint: the columns of \(\mat{U}\) and the rows of \(\mat{V}^{*}\) are orthonormal, while \(\mat{T}\) and \(\mat{W}\) need not be. \Cref{prop:orthonormal-singular-values} shows that this constraint is sufficient for the singular values to transfer.}
  \label{fig:shapes}
\end{figure}

\subsection{The reduced matrices place no upper bound on the norm}

This is not a matter of small discrepancies. Both reduced matrices can be held completely fixed while the norm of \(\mat{A}\) grows without bound.

\begin{example}[Fixed spectrum, unbounded norm]
\label{ex:unbounded-norm}
Take \(m=2\), \(n=1\), and
\[
  \mat{T}=\begin{pmatrix}1\\0\end{pmatrix},
  \qquad
  \mat{W}=\begin{pmatrix}1&\kappa\end{pmatrix},
  \qquad
  \mat{S}=\begin{pmatrix}1\end{pmatrix},
\]
with \(\kappa\) real. Then \(\mat{W}\mat{T}=\I_1\), so
\[
  \mat{B}=\mat{C}=\mat{S}=\begin{pmatrix}1\end{pmatrix}
\]
for every \(\kappa\). The full matrix is
\[
  \mat{A}=\mat{T}\mat{S}\mat{W}
  =\begin{pmatrix}1&\kappa\\0&0\end{pmatrix},
\]
whose eigenvalues are \(1\) and \(0\), again for every \(\kappa\). Its singular values are
\[
  \sqrt{1+\kappa^{2}}
  \qquad\text{and}\qquad
  0.
\]
Either ordering gives them. Forming \(\mat{A}^{*}\mat{A}\) leads to a characteristic equation, while the reversed product
\[
  \mat{A}\mat{A}^{*}
  =\begin{pmatrix}1+\kappa^{2}&0\\0&0\end{pmatrix}
\]
is diagonal and gives the nonzero value by inspection. Choosing the convenient ordering is the same move the rest of this article makes with \(\mat{A}\) and \(\mat{B}\).
The reduced matrix reports the same thing at every \(\kappa\), while \(\norm{\mat{A}}_2\) grows without bound.
\end{example}

A reduced model therefore does not determine the conditioning of the full operator, which has to be established separately from the factors themselves. The reduced spectrum is not entirely silent: since \(\norm{\mat{A}}_2\) is at least the spectral radius, and that radius is shared with \(\mat{B}\), the reduced eigenvalues bound \(\norm{\mat{A}}_2\) from below. They place no bound on it from above, which is what \cref{ex:unbounded-norm} exploits.

\subsection{When the singular values do agree}

There is one important case in which singular values transfer exactly. The statement below is classical: it is the invariance of singular values under unitary transformation \parencite[Section~2.6]{horn2013matrix}, extended to a rectangular isometry. The proof is given anyway, because it is another application of \cref{cor:rectangular-characteristic}, this time to \(\mat{A}^{*}\mat{A}\) rather than to \(\mat{A}\).

\begin{proposition}[Orthonormal bases preserve singular values]
\label{prop:orthonormal-singular-values}
Suppose the columns of \(\mat{T}\) are orthonormal, so that \(\mat{T}^{*}\mat{T}=\I_n\), and take \(\mat{W}=\mat{T}^{*}\). Then \(\mat{B}=\mat{C}=\mat{S}\), and the nonzero singular values of \(\mat{A}=\mat{T}\mat{S}\mat{T}^{*}\) are exactly the nonzero singular values of \(\mat{S}\).
\end{proposition}

\begin{proof}
Because \(\mat{W}\mat{T}=\mat{T}^{*}\mat{T}=\I_n\), the matrix \(\mat{M}\) is the identity, so \(\mat{B}=\mat{C}=\mat{S}\). For the singular values,
\[
  \mat{A}^{*}\mat{A}
  =\mat{T}\mat{S}^{*}(\mat{T}^{*}\mat{T})\mat{S}\mat{T}^{*}
  =\mat{T}(\mat{S}^{*}\mat{S})\mat{T}^{*}.
\]
Apply \cref{cor:rectangular-characteristic} with
\[
  \mat{X}=\mat{T}(\mat{S}^{*}\mat{S}),
  \qquad
  \mat{Y}=\mat{T}^{*}.
\]
Then \(\mat{X}\mat{Y}=\mat{A}^{*}\mat{A}\) and
\(\mat{Y}\mat{X}=\mat{S}^{*}\mat{S}\), so the two have the same nonzero eigenvalues. Taking positive square roots gives the statement.
\end{proof}

\begin{remark}[Biorthogonality is not enough]
The condition \(\mat{W}\mat{T}=\I_n\) by itself does not give this conclusion, as \cref{ex:unbounded-norm} shows. It does force \(\mat{B}=\mat{C}=\mat{S}\). Two hypotheses share the work in \cref{prop:orthonormal-singular-values}. Orthonormal columns, \(\mat{T}^{*}\mat{T}=\I_n\), cancel the inner factor of \cref{eq:AstarA}, leaving \(\mat{A}^{*}\mat{A}=\mat{W}^{*}(\mat{S}^{*}\mat{S})\mat{W}\). The pairing \(\mat{W}=\mat{T}^{*}\) then collapses the outer factors: by \cref{cor:rectangular-characteristic} applied once more, the nonzero eigenvalues of \(\mat{W}^{*}(\mat{S}^{*}\mat{S})\mat{W}\) are those of \((\mat{S}^{*}\mat{S})(\mat{W}\mat{W}^{*})\), and \(\mat{W}=\mat{T}^{*}\) forces \(\mat{W}\mat{W}^{*}=\I_n\). Neither hypothesis can be dropped. Taking \(\mat{T}=(1,0)^{\mathsf T}\), \(\mat{W}=(1,3)\), and \(\mat{S}=(1)\) gives both \(\mat{T}^{*}\mat{T}=\I_1\) and \(\mat{W}\mat{T}=\I_1\), yet \(\mat{W}\mat{W}^{*}=10\), so the one nonzero singular value of \(\mat{A}\) is \(\sqrt{10}\) rather than \(1\). An oblique biorthogonal pair leaves \cref{eq:AstarA} with oblique outer factors.
\end{remark}

\subsection{Why this matters in practice}

Three consequences deserve attention.

Condition numbers do not transfer. The sensitivity of a full linear system cannot be read off a reduced matrix, and the same holds for least-squares problems, where sensitivity is governed by singular values rather than by eigenvalues.

Errors can grow when a reduced solution is lifted. That growth is measured by \(\norm{\mat{T}}\), a quantity invisible to both \(\mat{B}\) and \(\mat{C}\). \Cref{sec:limitations} returns to the point.

The reduction strategies surveyed in \cref{sec:projection} are themselves built from singular values. Proper orthogonal decomposition selects directions by the singular values of a snapshot matrix, and balanced truncation ranks states by Hankel singular values \parencite{antoulas2005approximation}. Neither is a spectral construction. This is one reason orthonormal bases are preferred whenever a reduced model has to be trusted quantitatively.

\section{Reduced algebraic systems}
\label{sec:algebraic-systems}

\subsection{Applying the full operator without forming it}

Given \(\vect{x}\in\F^m\), compute \(\mat{A}\vect{x}\) in three steps:
\begin{align}
  \vect{q}&=\mat{W}\vect{x},
  \label{eq:apply-step1}\\
  \vect{r}&=\mat{S}\vect{q},
  \label{eq:apply-step2}\\
  \mat{A}\vect{x}&=\mat{T}\vect{r}.
  \label{eq:apply-step3}
\end{align}
No \(m\times m\) matrix is required. For dense factors, the cost is on the order of \(mn+n^2\) operations rather than \(m^2\). If \(\mat{T}\) or \(\mat{W}\) is sparse or has fast transforms, the savings can be larger.

\subsection{Powers of the large matrix}

\begin{proposition}[Powers through the reduced matrices]
\label{prop:powers}
For every integer \(k\geq1\),
\begin{equation}
  \mat{A}^k
  =\mat{T}\mat{S}\mat{B}^{k-1}\mat{W}
  =\mat{T}\mat{C}^{k-1}\mat{S}\mat{W}.
  \label{eq:powers-reduced}
\end{equation}
\end{proposition}

\begin{proof}
For the first identity, the case \(k=1\) is \cref{eq:A-factorization}. If it holds for \(k\), then
\begin{align*}
  \mat{A}^{k+1}
  &=\mat{A}^k\mat{A}\\
  &=\mat{T}\mat{S}\mat{B}^{k-1}\mat{W}
    \mat{T}\mat{S}\mat{W}\\
  &=\mat{T}\mat{S}\mat{B}^{k-1}
    (\mat{W}\mat{T})\mat{S}\mat{W}\\
  &=\mat{T}\mat{S}\mat{B}^{k}\mat{W}.
\end{align*}
The proof of the second identity is analogous, grouping
\(\mat{S}\mat{W}\mat{T}=\mat{C}\).
\end{proof}

Repeated application of \(\mat{A}\) therefore reduces to repeated application of either \(\mat{B}\) or \(\mat{C}\). For example,
\begin{equation}
  \mat{A}^k\vect{x}
  =\mat{T}\mat{S}\mat{B}^{k-1}(\mat{W}\vect{x}).
  \label{eq:power-vector}
\end{equation}
The full dimension appears only in the first extraction and the final lifting.

\subsection{Shifted linear systems and the resolvent}

The matrix \(\mat{A}\) is singular whenever \(m>n\), so a system
\(\mat{A}\vect{x}=\vect{b}\) may have no solution or many solutions. A more useful problem is the shifted system
\begin{equation}
  (\mu\I_m-\mat{A})\vect{x}=\vect{b},
  \label{eq:shifted-system}
\end{equation}
where \(\mu\neq0\).

\begin{theorem}[Reduced resolvent formula]
\label{thm:resolvent}
Assume \(\mu\neq0\) and \(\mu\I_n-\mat{B}\) is invertible. Then
\begin{equation}
  (\mu\I_m-\mat{A})^{-1}
  =
  \frac{1}{\mu}\I_m
  +
  \frac{1}{\mu}
  \mat{T}\mat{S}
  (\mu\I_n-\mat{B})^{-1}
  \mat{W}.
  \label{eq:reduced-resolvent}
\end{equation}
\end{theorem}

\begin{proof}
Let
\[
  \mat{R}
  =
  \frac{1}{\mu}\I_m
  +
  \frac{1}{\mu}
  \mat{T}\mat{S}
  (\mu\I_n-\mat{B})^{-1}\mat{W}.
\]
Then
\begin{align*}
  (\mu\I_m-\mat{A})\mat{R}
  &=\I_m-\frac{1}{\mu}\mat{A}
    +\mat{T}\mat{S}(\mu\I_n-\mat{B})^{-1}\mat{W}\\
  &\quad
    -\frac{1}{\mu}
     \mat{T}\mat{S}\mat{W}\mat{T}\mat{S}
     (\mu\I_n-\mat{B})^{-1}\mat{W}\\
  &=\I_m-\frac{1}{\mu}\mat{T}\mat{S}\mat{W}
    +\mat{T}\mat{S}
     \left[\I_n-\frac{1}{\mu}\mat{B}\right]
     (\mu\I_n-\mat{B})^{-1}\mat{W}\\
  &=\I_m-\frac{1}{\mu}\mat{T}\mat{S}\mat{W}
    +\frac{1}{\mu}\mat{T}\mat{S}\mat{W}\\
  &=\I_m.
\end{align*}
Because \(\mu\I_m-\mat{A}\) is square, a right inverse is the inverse, so \(\mat{R}=(\mu\I_m-\mat{A})^{-1}\). The formula is a low-rank case of the Sherman--Morrison--Woodbury identity, the standard tool for inverting a matrix that differs from an easily inverted one by a low-rank term.
\end{proof}

The theorem converts \cref{eq:shifted-system} into the following procedure:
\begin{enumerate}[label=\arabic*.]
  \item Form the reduced right-hand side \(\vect{g}=\mat{W}\vect{b}\).
  \item Solve
  \begin{equation}
    (\mu\I_n-\mat{B})\vect{z}=\vect{g}.
    \label{eq:reduced-shifted-solve}
  \end{equation}
  \item Reconstruct
  \begin{equation}
    \vect{x}
    =\frac{1}{\mu}\vect{b}
     +\frac{1}{\mu}\mat{T}\mat{S}\vect{z}.
    \label{eq:reconstruct-shifted-solve}
  \end{equation}
\end{enumerate}

\begin{example}[Solving a shifted system through dimension two]
\label{ex:shifted-solve}
Use the matrices from \cref{ex:three-by-three}, take \(\mu=1\), and let
\[
  \vect{b}=
  \begin{pmatrix}1\\2\\0\end{pmatrix}.
\]
The reduced right-hand side is
\[
  \mat{W}\vect{b}
  =
  \begin{pmatrix}1\\2\end{pmatrix}.
\]
Since
\[
  \I_2-\mat{B}
  =
  \begin{pmatrix}
    3&0\\
    2&7
  \end{pmatrix},
\]
we solve
\[
  \begin{pmatrix}
    3&0\\
    2&7
  \end{pmatrix}
  \vect{z}
  =
  \begin{pmatrix}1\\2\end{pmatrix}
\]
to obtain
\begin{equation}
  \vect{z}
  =
  \begin{pmatrix}
    1/3\\
    4/21
  \end{pmatrix}.
  \label{eq:example-z}
\end{equation}
Then
\begin{align*}
  \vect{x}
  &=\vect{b}+\mat{T}\mat{S}\vect{z}\\
  &=
  \begin{pmatrix}
    1/3\\
    10/7\\
    -26/21
  \end{pmatrix}.
\end{align*}
Direct substitution verifies
\[
  (\I_3-\mat{A})\vect{x}=\vect{b}.
\]
The only linear solve was \(2\times2\).
\end{example}

\subsection{Matrix functions}

The same power identity reduces analytic functions of \(\mat{A}\). For background on matrix functions, see \textcite{higham2008functions}.

\begin{proposition}[Analytic matrix functions]
\label{prop:matrix-functions}
Suppose \(f\) has a power series about the origin whose radius of convergence exceeds the spectral radius of \(\mat{A}\), equivalently of \(\mat{B}\), since the two share their nonzero eigenvalues. Define
\begin{equation}
  g(z)=
  \begin{cases}
    \dfrac{f(z)-f(0)}{z},&z\neq0,\\[6pt]
    f'(0),&z=0.
  \end{cases}
  \label{eq:g-definition}
\end{equation}
Then
\begin{equation}
  f(\mat{A})
  =f(0)\I_m+\mat{T}\mat{S}g(\mat{B})\mat{W}.
  \label{eq:matrix-function-reduction}
\end{equation}
\end{proposition}

\begin{proof}
Write \(f(z)=\sum_{k=0}^{\infty}a_kz^k\). Then
\[
  f(\mat{A})
  =a_0\I_m+
   \sum_{k=1}^{\infty}a_k\mat{A}^k.
\]
Using \cref{eq:powers-reduced},
\begin{align*}
  f(\mat{A})
  &=a_0\I_m+
    \mat{T}\mat{S}
    \left(
      \sum_{k=1}^{\infty}a_k\mat{B}^{k-1}
    \right)
    \mat{W}\\
  &=f(0)\I_m+\mat{T}\mat{S}g(\mat{B})\mat{W}.
\end{align*}
\end{proof}

For the matrix exponential, define
\begin{equation}
  \varphi_1(z)=
  \begin{cases}
    \dfrac{\e^z-1}{z},&z\neq0,\\[6pt]
    1,&z=0.
  \end{cases}
  \label{eq:phi1}
\end{equation}
Then
\begin{equation}
  \e^{t\mat{A}}
  =\I_m+t\mat{T}\mat{S}\varphi_1(t\mat{B})\mat{W}.
  \label{eq:exponential-reduction}
\end{equation}
This formula will reappear in the ODE analysis.

\section{Exact reduction of first-order linear ODEs}
\label{sec:ode-reduction}

Consider the autonomous system
\begin{equation}
  \vect{x}'(t)=\mat{A}\vect{x}(t),
  \qquad
  \mat{A}=\mat{T}\mat{S}\mat{W},
  \qquad
  \vect{x}(0)=\vect{x}_0.
  \label{eq:full-ode}
\end{equation}
There are two natural reduced variables. They lead to \(\mat{B}\) and \(\mat{C}\), respectively.

\subsection{Reduction by extracted observables}

Define
\begin{equation}
  \vect{q}(t)=\mat{W}\vect{x}(t)\in\F^n.
  \label{eq:q-definition}
\end{equation}
Then
\begin{align}
  \vect{q}'(t)
  &=\mat{W}\vect{x}'(t)\notag\\
  &=\mat{W}\mat{T}\mat{S}\mat{W}\vect{x}(t)\notag\\
  &=\mat{B}\vect{q}(t).
  \label{eq:q-reduced-ode}
\end{align}
The initial condition is
\begin{equation}
  \vect{q}(0)=\mat{W}\vect{x}_0.
  \label{eq:q-initial}
\end{equation}
\Cref{eq:q-reduced-ode} is an exact closed system for the extracted variables, and it holds for every initial state \(\vect{x}_0\) with no restriction whatever.

Once \(\vect{q}(t)\) is known, the full derivative is
\begin{equation}
  \vect{x}'(t)=\mat{T}\mat{S}\vect{q}(t),
  \label{eq:xprime-from-q}
\end{equation}
so
\begin{equation}
  \vect{x}(t)
  =\vect{x}_0+
   \int_0^t\mat{T}\mat{S}\vect{q}(s)\,\dd s.
  \label{eq:x-reconstruction-integral}
\end{equation}
Because
\[
  \vect{q}(t)=\e^{t\mat{B}}\mat{W}\vect{x}_0,
\]
this integral reproduces \cref{eq:exponential-reduction}.

\subsection{Reduction by lifted coordinates}

Suppose \(\mat{T}\) has full column rank and the initial condition belongs to \(\range(\mat{T})\). Write
\begin{equation}
  \vect{x}_0=\mat{T}\vect{y}_0.
  \label{eq:x0-in-range}
\end{equation}
Because \(\range(\mat{T})\) is invariant under \(\mat{A}\) by \cref{cor:invariant-range}, every term of the series for \(\e^{t\mat{A}}\vect{x}_0\) lies in that subspace, so the solution remains there for all \(t\). Write
\begin{equation}
  \vect{x}(t)=\mat{T}\vect{y}(t).
  \label{eq:x-Ty}
\end{equation}
Substitution into \cref{eq:full-ode} gives
\[
  \mat{T}\vect{y}'(t)
  =\mat{A}\mat{T}\vect{y}(t)
  =\mat{T}\mat{C}\vect{y}(t).
\]
Because \(\mat{T}\) has full column rank,
\begin{equation}
  \vect{y}'(t)=\mat{C}\vect{y}(t),
  \qquad
  \vect{y}(0)=\vect{y}_0.
  \label{eq:y-reduced-ode}
\end{equation}
The full solution is reconstructed simply by
\begin{equation}
  \vect{x}(t)=\mat{T}\e^{t\mat{C}}\vect{y}_0.
  \label{eq:x-from-y}
\end{equation}

The two reduced states have different meanings:
\begin{equation}
  \vect{q}=\mat{W}\vect{x}
  \quad\text{contains extracted observables,}
  \qquad
  \vect{y}
  \quad\text{contains coefficients in the columns of }\mat{T}.
  \label{eq:q-y-meaning}
\end{equation}
If \(\vect{x}=\mat{T}\vect{y}\), then
\begin{equation}
  \vect{q}=\mat{M}\vect{y}.
  \label{eq:q-My}
\end{equation}
When \(\mat{M}\) is invertible, the two reduced coordinates are related by a change of basis, and \(\mat{B}\) and \(\mat{C}\) are similar.

The two reductions also differ in what they require. The observable reduction \cref{eq:q-reduced-ode} needs no assumption on \(\vect{x}_0\), because \(\vect{q}=\mat{W}\vect{x}\) is defined for every full state. The lifted reduction \cref{eq:y-reduced-ode} applies only when \(\vect{x}_0\) lies in \(\range(\mat{T})\). Outside that subspace there is no \(\vect{y}_0\) to start from.

\subsection{The biorthogonal decomposition of a trajectory}

If \(\mat{W}\mat{T}=\I_n\), then \(\mat{P}=\mat{T}\mat{W}\) is a projection. Decompose
\begin{equation}
  \vect{x}_0=\mat{P}\vect{x}_0+(\I_m-\mat{P})\vect{x}_0.
  \label{eq:x0-active-inactive}
\end{equation}
The second term, \((\I_m-\mat{P})\vect{x}_0\), lies in \(\nullsp(\mat{W})\), so it is annihilated by \(\mat{A}\) and holds its initial value for all \(t\). It is this whole component that stays fixed, not any single coordinate of \(\vect{x}\). The active component evolves in \(\range(\mat{T})\). The trajectory splits accordingly:
\begin{equation}
  \vect{x}(t)
  =
  (\I_m-\mat{P})\vect{x}_0
  +\mat{T}\e^{t\mat{S}}\mat{W}\vect{x}_0.
  \label{eq:biorthogonal-solution}
\end{equation}
This formula cleanly separates inactive directions from reduced dynamics.

\section{A worked ODE example}
\label{sec:worked-ode}

\begin{example}[A three-dimensional ODE solved in dimension two]
\label{ex:ode-three-to-two}
Let
\begin{equation}
  \mat{T}=
  \begin{pmatrix}
    1&0\\
    0&1\\
    1&1
  \end{pmatrix},
  \qquad
  \mat{W}=
  \begin{pmatrix}
    1&0&0\\
    0&1&0
  \end{pmatrix},
  \qquad
  \mat{S}=
  \begin{pmatrix}
    -1&1\\
    0&-2
  \end{pmatrix}.
  \label{eq:ode-example-factors}
\end{equation}
Here
\begin{equation}
  \mat{W}\mat{T}=\I_2,
  \label{eq:ode-example-WT}
\end{equation}
so \(\mat{B}=\mat{C}=\mat{S}\). The full matrix is
\begin{equation}
  \mat{A}
  =\mat{T}\mat{S}\mat{W}
  =
  \begin{pmatrix}
    -1&1&0\\
    0&-2&0\\
    -1&-1&0
  \end{pmatrix}.
  \label{eq:ode-example-A}
\end{equation}
Consider
\begin{equation}
  \vect{x}'=\mat{A}\vect{x},
  \qquad
  \vect{x}(0)=\mat{T}
  \begin{pmatrix}a\\b\end{pmatrix}
  =
  \begin{pmatrix}a\\b\\a+b\end{pmatrix}.
  \label{eq:ode-example-full}
\end{equation}
Because the initial state lies in \(\range(\mat{T})\), write
\(\vect{x}=\mat{T}\vect{y}\). The reduced problem is
\begin{equation}
  \vect{y}'=\mat{S}\vect{y},
  \qquad
  \vect{y}(0)=
  \begin{pmatrix}a\\b\end{pmatrix},
  \label{eq:ode-example-reduced}
\end{equation}
or, componentwise,
\begin{align}
  y_1'&=-y_1+y_2,
  \label{eq:ode-example-y1}\\
  y_2'&=-2y_2.
  \label{eq:ode-example-y2}
\end{align}
The second equation gives
\begin{equation}
  y_2(t)=b\e^{-2t}.
  \label{eq:ode-example-y2-sol}
\end{equation}
Using an integrating factor in the first equation, and fixing the constant so that \(y_1(0)=a\),
\begin{equation}
  y_1(t)=(a+b)\e^{-t}-b\e^{-2t}.
  \label{eq:ode-example-y1-sol}
\end{equation}
Lifting with \(\mat{T}\) gives
\begin{equation}
  \vect{x}(t)
  =
  \begin{pmatrix}
    (a+b)\e^{-t}-b\e^{-2t}\\
    b\e^{-2t}\\
    (a+b)\e^{-t}
  \end{pmatrix}.
  \label{eq:ode-example-x-sol}
\end{equation}
Direct differentiation verifies \(\vect{x}'=\mat{A}\vect{x}\).
\end{example}

\subsection{Spectral interpretation}

The eigenvalues of \(\mat{S}\) are \(-1\) and \(-2\). Sylvester's theorem predicts that the nonzero eigenvalues of \(\mat{A}\) are \(-1\) and \(-2\). Because \(m-n=1\), the full matrix has one additional zero eigenvalue:
\begin{equation}
  \spec(\mat{A})=\{0,-1,-2\}.
  \label{eq:ode-example-spectrum}
\end{equation}
The zero eigenvector \(\vect{e}_3=(0,0,1)^{\mathsf T}\) lies outside \(\range(\mat{T})\) and belongs to \(\nullsp(\mat{W})\). It represents an inactive component that remains constant. The chosen initial condition in \cref{eq:ode-example-full} has no inactive component, so only the two decaying modes appear in \cref{eq:ode-example-x-sol}.

\subsection{A general initial condition}

For an arbitrary
\[
  \vect{x}_0=
  \begin{pmatrix}x_{10}\\x_{20}\\x_{30}\end{pmatrix},
\]
the projector is
\begin{equation}
  \mat{P}=\mat{T}\mat{W}
  =
  \begin{pmatrix}
    1&0&0\\
    0&1&0\\
    1&1&0
  \end{pmatrix}.
  \label{eq:ode-example-P}
\end{equation}
The inactive part is
\begin{equation}
  (\I_3-\mat{P})\vect{x}_0
  =
  \begin{pmatrix}
    0\\0\\x_{30}-x_{10}-x_{20}
  \end{pmatrix}.
  \label{eq:ode-example-inactive}
\end{equation}
It remains constant, while the active variables evolve according to \(\mat{S}\). This gives the full solution
\begin{equation}
  \vect{x}(t)
  =(\I_3-\mat{P})\vect{x}_0
  +\mat{T}\e^{t\mat{S}}\mat{W}\vect{x}_0.
  \label{eq:ode-example-general-solution}
\end{equation}

\section{Shifted systems, forcing, and stability}
\label{sec:shifted-stability}

A pure low-rank generator \(\mat{A}=\mat{T}\mat{S}\mat{W}\) has at least \(m-n\) zero eigenvalues. Therefore the origin cannot be asymptotically stable for
\(\vect{x}'=\mat{A}\vect{x}\) when \(m>n\). In applications, the low-rank interaction often appears together with full-space decay:
\begin{equation}
  \vect{x}'
  =(-\gamma\I_m+\mat{T}\mat{S}\mat{W})\vect{x},
  \qquad \gamma>0.
  \label{eq:shifted-ode}
\end{equation}

\begin{theorem}[Spectrum of a shifted low-rank system]
\label{thm:shifted-spectrum}
Let
\[
  \mat{J}=-\gamma\I_m+\mat{A}.
\]
Then
\begin{equation}
  \det(\lambda\I_m-\mat{J})
  =
  (\lambda+\gamma)^{m-n}
  \det\!\left((\lambda+\gamma)\I_n-\mat{B}\right).
  \label{eq:shifted-characteristic}
\end{equation}
Hence the spectrum consists of
\begin{equation}
  -\gamma
  \quad\text{with multiplicity at least }m-n,
  \label{eq:shifted-inactive-eigs}
\end{equation}
and
\begin{equation}
  -\gamma+\lambda_j(\mat{B}),
  \qquad j=1,\dots,n,
  \label{eq:shifted-active-eigs}
\end{equation}
with multiplicities inherited from \(\mat{B}\). The two lists overlap when \(\mat{B}\) is singular, since each zero eigenvalue of \(\mat{B}\) contributes a further copy of \(-\gamma\). The exact multiplicity of \(-\gamma\) is \(m-n\) plus the algebraic multiplicity of zero in \(\mat{B}\).
\end{theorem}

\begin{proof}
Use \cref{eq:three-factor-characteristic} with \(\lambda+\gamma\) in place of \(\lambda\):
\begin{align*}
  \det(\lambda\I_m-\mat{J})
  &=\det((\lambda+\gamma)\I_m-\mat{A})\\
  &=(\lambda+\gamma)^{m-n}
    \det((\lambda+\gamma)\I_n-\mat{B}).
\end{align*}
\end{proof}

\begin{corollary}[Stability criterion]
\label{cor:stability-criterion}
Assume \(\gamma>0\). The system in \cref{eq:shifted-ode} is asymptotically stable if and only if
\begin{equation}
  \max_{1\leq j\leq n}
  \RePart\lambda_j(\mat{B})<\gamma.
  \label{eq:stability-criterion}
\end{equation}
\end{corollary}

\begin{proof}
By \cref{thm:shifted-spectrum} the eigenvalues of \(\mat{J}\) are \(-\gamma\) together with \(-\gamma+\lambda_j(\mat{B})\). Because \(\gamma>0\), the first has negative real part. The others do exactly when \(\RePart\lambda_j(\mat{B})<\gamma\) for every \(j\).
\end{proof}

The criterion reduces an \(m\)-dimensional stability test to an \(n\)-dimensional eigenvalue calculation. It settles asymptotic behavior as \(t\to\infty\) and nothing beyond that. \Cref{sec:limitations} returns to what eigenvalues leave undetermined about transient growth.

\begin{example}[Stability from the reduced matrix]
Use \(\mat{B}\) from \cref{eq:example1-B}, whose eigenvalues are \(-2\) and \(-6\), and take \(\gamma=1\). Then the full shifted matrix
\[
  \mat{J}=-\I_3+\mat{A}
\]
has eigenvalues
\[
  -1,\qquad -3,\qquad -7.
\]
The eigenvalue \(-1\) comes from the inactive dimension; the other two are obtained by shifting the reduced eigenvalues by \(-1\).
\end{example}

\subsection{Forced systems}

Consider
\begin{equation}
  \vect{x}'
  =(-\gamma\I_m+\mat{A})\vect{x}+\vect{f}(t).
  \label{eq:forced-full-ode}
\end{equation}
The extracted variable \(\vect{q}=\mat{W}\vect{x}\) satisfies
\begin{equation}
  \vect{q}'
  =(-\gamma\I_n+\mat{B})\vect{q}
   +\mat{W}\vect{f}(t).
  \label{eq:forced-reduced-q}
\end{equation}
This equation is exact for every forcing function. Once \(\vect{q}\) is found, the full state satisfies
\begin{equation}
  \vect{x}'
  =-\gamma\vect{x}
   +\mat{T}\mat{S}\vect{q}
   +\vect{f}(t).
  \label{eq:forced-reconstruct}
\end{equation}
Suppose \(\mat{T}\) has full column rank, as in the lifted-coordinate reduction of \cref{sec:ode-reduction}, so that \(\mat{T}\) is injective. If the forcing also lies in the lifted subspace, \(\vect{f}(t)=\mat{T}\vect{h}(t)\), and the initial condition belongs to \(\range(\mat{T})\), then cancelling \(\mat{T}\) is legitimate and the lifted coordinates satisfy
\begin{equation}
  \vect{y}'
  =(-\gamma\I_n+\mat{C})\vect{y}+\vect{h}(t).
  \label{eq:forced-reduced-y}
\end{equation}

\section{From exact factorization to projection-based model reduction}
\label{sec:projection}

The preceding reductions are exact because the full matrix is assumed to factor through an \(n\)-dimensional space. Most model-reduction problems begin differently: a general large matrix is given, and one seeks a useful low-dimensional approximation. Projection methods provide the bridge \parencite{antoulas2005approximation,benner2015survey,quarteroni2016reduced}.

\subsection{Trial and test spaces}

Consider the full system
\begin{equation}
  \vect{x}'=\mat{L}\vect{x},
  \qquad
  \mat{L}\in\F^{m\times m}.
  \label{eq:general-full-ode}
\end{equation}
Choose a trial basis
\begin{equation}
  \mat{T}\in\F^{m\times n},
  \qquad
  \rank(\mat{T})=n,
  \label{eq:trial-basis}
\end{equation}
and approximate
\begin{equation}
  \vect{x}(t)\approx\mat{T}\vect{y}(t).
  \label{eq:trial-approximation}
\end{equation}
Choose a test map
\begin{equation}
  \mat{W}\in\F^{n\times m}
  \label{eq:test-map}
\end{equation}
that is usually normalized so that
\begin{equation}
  \mat{W}\mat{T}=\I_n.
  \label{eq:petrov-biorthogonality}
\end{equation}
The residual of the approximation is
\begin{equation}
  \vect{r}
  =\mat{T}\vect{y}'-\mat{L}\mat{T}\vect{y}.
  \label{eq:projection-residual}
\end{equation}
A Petrov--Galerkin condition requires
\begin{equation}
  \mat{W}\vect{r}=\vect{0}.
  \label{eq:petrov-condition}
\end{equation}
Using \cref{eq:petrov-biorthogonality} gives
\begin{equation}
  \vect{y}'=\mat{L}_r\vect{y},
  \qquad
  \mat{L}_r=\mat{W}\mat{L}\mat{T}.
  \label{eq:projected-operator}
\end{equation}

If \(\mat{T}\) has orthonormal columns and \(\mat{W}=\mat{T}^{*}\), this is a Galerkin projection. More general choices yield Petrov--Galerkin methods \parencite{quarteroni2016reduced,benner2015survey}.

\subsection{The lifted reduced operator}

The reduced matrix acts in \(\F^n\). To compare it with a full-space operator, define
\begin{equation}
  \widetilde{\mat{L}}
  =\mat{T}\mat{L}_r\mat{W}.
  \label{eq:lifted-reduced-operator}
\end{equation}
This matrix has the exact factor form studied earlier, with \(\mat{S}=\mat{L}_r\). Because \(\mat{W}\mat{T}=\I_n\), Sylvester's theorem gives
\begin{equation}
  \det(\lambda\I_m-\widetilde{\mat{L}})
  =\lambda^{m-n}
   \det(\lambda\I_n-\mat{L}_r).
  \label{eq:lifted-reduced-spectrum}
\end{equation}
The nonzero eigenvalues of the lifted approximation are exactly the nonzero eigenvalues of the reduced matrix, with the same multiplicities. If \(\mat{L}_r\) is singular its zero eigenvalue is not among them.

The statement is useful, but it is easy to over-read:
\begin{equation}
  \spec_{\neq0}(\widetilde{\mat{L}})
  =\spec_{\neq0}(\mat{L}_r)
  \label{eq:exact-lifted-spectrum}
\end{equation}
is an exact algebraic identity, whereas
\begin{equation}
  \spec(\mat{L}_r)
  \approx\spec(\mat{L})
  \label{eq:approx-full-spectrum}
\end{equation}
is an approximation claim that requires justification. The first is free: it holds for any \(\mat{T}\) and \(\mat{W}\) with \(\mat{W}\mat{T}=\I_n\), whatever subspace they happen to describe. The second carries the entire modeling risk, and Sylvester's theorem says nothing about it.

\subsection{When the reduction is exact}

\begin{theorem}[Exact invariant-subspace reduction]
\label{thm:invariant-subspace-reduction}
Suppose \(\mat{W}\mat{T}=\I_n\) and there is a matrix \(\mat{S}\in\F^{n\times n}\) such that
\begin{equation}
  \mat{L}\mat{T}=\mat{T}\mat{S}.
  \label{eq:invariance-condition}
\end{equation}
Then
\begin{equation}
  \mat{W}\mat{L}\mat{T}=\mat{S}.
  \label{eq:projected-S}
\end{equation}
For every initial condition \(\vect{x}_0=\mat{T}\vect{y}_0\), the full solution of
\(\vect{x}'=\mat{L}\vect{x}\) is
\begin{equation}
  \vect{x}(t)=\mat{T}\e^{t\mat{S}}\vect{y}_0.
  \label{eq:invariant-exact-solution}
\end{equation}
\end{theorem}

\begin{proof}
The proof has two steps. The first is a short computation. The second constructs a
candidate solution, verifies that it satisfies the same initial-value problem as
\(\vect{x}\), and then appeals to uniqueness. That pattern is worth recognizing: it
identifies \(\vect{x}\) without ever integrating the full system.

Multiplying \cref{eq:invariance-condition} by \(\mat{W}\) gives
\[
  \mat{W}\mat{L}\mat{T}
  =\mat{W}\mat{T}\mat{S}
  =\mat{S}.
\]
If \(\vect{x}_0=\mat{T}\vect{y}_0\), set
\[
  \vect{z}(t)=\mat{T}\e^{t\mat{S}}\vect{y}_0.
\]
Then \(\vect{z}(0)=\mat{T}\vect{y}_0=\vect{x}_0\), and
\[
  \vect{z}'(t)
  =\mat{T}\mat{S}\e^{t\mat{S}}\vect{y}_0
  =\mat{L}\mat{T}\e^{t\mat{S}}\vect{y}_0
  =\mat{L}\vect{z}(t),
\]
where the middle step is \cref{eq:invariance-condition}. So \(\vect{z}\) solves the same initial-value problem as \(\vect{x}\), and the solution of a linear initial-value problem is unique, giving \cref{eq:invariant-exact-solution}.
\end{proof}

\subsection{Measuring the defect of an approximate subspace}

For a projected matrix
\[
  \mat{L}_r=\mat{W}\mat{L}\mat{T},
\]
define the invariance residual
\begin{equation}
  \mat{R}=\mat{L}\mat{T}-\mat{T}\mat{L}_r.
  \label{eq:invariance-residual}
\end{equation}
If \(\mat{R}=\mat{0}\), the trial space is invariant and the reduction is exact on that space. If \(\norm{\mat{R}}\) is small, the trial space is nearly invariant in the chosen norm. This residual is one practical check on a spectral reduced model, though trajectory and output accuracy may require more specialized error estimates.

\begin{example}[A projection that preserves the algebra and loses the dynamics]
\label{ex:projection-failure}
Keep the trial and test maps of \cref{ex:ode-three-to-two},
\[
  \mat{T}=
  \begin{pmatrix}
    1&0\\
    0&1\\
    1&1
  \end{pmatrix},
  \qquad
  \mat{W}=
  \begin{pmatrix}
    1&0&0\\
    0&1&0
  \end{pmatrix},
  \qquad
  \mat{W}\mat{T}=\I_2,
\]
and replace the exact factorization by a general operator that does not factor through \(\range(\mat{T})\):
\begin{equation}
  \mat{L}=
  \begin{pmatrix}
    -1&-1&-1\\
     0&-2&-1\\
     0& 0& 1
  \end{pmatrix},
  \qquad
  \spec(\mat{L})=\{-1,-2,1\}.
  \label{eq:projection-example-L}
\end{equation}
The matrix is triangular, so its eigenvalues are the diagonal entries, and the full system \(\vect{x}'=\mat{L}\vect{x}\) has one growing mode.

The projected matrix and its characteristic polynomial are
\begin{equation}
  \mat{L}_r
  =\mat{W}\mat{L}\mat{T}
  =
  \begin{pmatrix}
    -2&-2\\
    -1&-3
  \end{pmatrix},
  \qquad
  \det(\lambda\I_2-\mat{L}_r)
  =(\lambda+1)(\lambda+4).
  \label{eq:projection-example-Lr}
\end{equation}
Lifting the reduced matrix back into the full space gives
\begin{equation}
  \widetilde{\mat{L}}
  =\mat{T}\mat{L}_r\mat{W}
  =
  \begin{pmatrix}
    -2&-2&0\\
    -1&-3&0\\
    -3&-5&0
  \end{pmatrix},
  \qquad
  \det(\lambda\I_3-\widetilde{\mat{L}})
  =\lambda(\lambda+1)(\lambda+4),
  \label{eq:projection-example-lifted}
\end{equation}
which is \cref{eq:lifted-reduced-spectrum} with \(m-n=1\). The algebra delivers what it promised.

The modeling claim does not. Against \(\spec(\mat{L})=\{-1,-2,1\}\), the reduced matrix keeps one eigenvalue, introduces \(-4\), and loses \(1\) altogether. A reader who stopped at \cref{eq:exact-lifted-spectrum} would conclude that every mode decays, when one of them grows. The habit built in \cref{sec:shifted-stability}, where a reduced spectrum settles asymptotic stability, was earned there by an exact factorization and is not available here.

The invariance residual records the gap. From \cref{eq:invariance-residual},
\begin{equation}
  \mat{R}
  =\mat{L}\mat{T}-\mat{T}\mat{L}_r
  =
  \begin{pmatrix}
    0&0\\
    0&0\\
    4&6
  \end{pmatrix},
  \qquad
  \norm{\mat{R}}_{F}=2\sqrt{13},
  \label{eq:projection-example-residual}
\end{equation}
of rank one and supported entirely in the third component. The eigenvector of \(\mat{L}\) for \(\lambda=1\) is a multiple of \((-1,-1,3)^{\mathsf T}\), while \(\range(\mat{T})\) is the plane \(x_3=x_1+x_2\). The residual points at the one direction the trial space cannot represent.

Changing only the trial space repairs the reduction. With
\[
  \mat{T}'=
  \begin{pmatrix}
    1&0\\
    0&1\\
    0&0
  \end{pmatrix},
  \qquad
  \mat{W}\mat{T}'=\I_2,
\]
the same construction gives
\begin{equation}
  \mat{L}_r'
  =\mat{W}\mat{L}\mat{T}'
  =
  \begin{pmatrix}
    -1&-1\\
     0&-2
  \end{pmatrix},
  \qquad
  \mat{R}'=\mat{0},
  \label{eq:projection-example-good}
\end{equation}
with \(\spec(\mat{L}_r')=\{-1,-2\}\), two exact eigenvalues of \(\mat{L}\). \Cref{eq:invariance-condition} holds, so \cref{thm:invariant-subspace-reduction} applies verbatim. The theorem was the same in both calculations. What changed was the subspace.

One caution survives the repair. The corrected trial space is invariant, and it still omits the growing mode, which lies outside \(\range(\mat{T}')\). Exactness on a subspace is a statement about that subspace, not about the behavior a reduced model was built to capture.
\end{example}

\subsection{How reduced spaces are chosen}

The theorem does not choose \(\mat{T}\) and \(\mat{W}\). Common strategies include:
\begin{itemize}
  \item eigenvectors or invariant subspaces associated with important modes \parencite{golub2013matrix};
  \item proper orthogonal decomposition from state snapshots \parencite{benner2015survey};
  \item Krylov subspaces designed to match moments or transfer functions \parencite{antoulas2005approximation};
  \item balanced truncation based on controllability and observability \parencite{antoulas2005approximation};
  \item reduced basis methods for parameter-dependent equations \parencite{quarteroni2016reduced}.
\end{itemize}
These methods address the modeling question: which low-dimensional subspace retains the behavior of interest? Sylvester's theorem addresses a different algebraic question: once an operator is written in lifted reduced form, how are its nonzero eigenvalues related to those of the reduced matrix?

\section{Computational workflow}
\label{sec:computational-workflow}

A practical implementation should preserve the factorization rather than assemble \(\mat{A}\).

\subsection{Reduced spectral analysis}

Given \(\mat{T}\), \(\mat{S}\), and \(\mat{W}\):
\begin{enumerate}[label=\arabic*.]
  \item Form
  \[
    \mat{M}=\mat{W}\mat{T}.
  \]
  \item Form either
  \[
    \mat{B}=\mat{M}\mat{S}
    \quad\text{or}\quad
    \mat{C}=\mat{S}\mat{M}.
  \]
  \item Compute the eigenvalues of the \(n\times n\) reduced matrix.
  \item Interpret these \(n\) values as the eigenvalues of \(\mat{A}\) inherited from \(\mat{B}\), and append exactly \(m-n\) further zeros for the inactive dimensions. The total multiplicity of the zero eigenvalue is \(m-n\) plus the algebraic multiplicity of zero in \(\mat{B}\).
  \item Reconstruct selected full eigenvectors using \(\mat{T}\vect{v}\) from eigenvectors of \(\mat{C}\), or \(\mat{T}\mat{S}\vect{u}\) from eigenvectors of \(\mat{B}\).
  \item Check the full residual without forming \(\mat{A}\):
  \begin{equation}
    \rho
    =
    \frac{
      \norm{\mat{T}\mat{S}(\mat{W}\vect{x})-\lambda\vect{x}}
    }{
      \norm{\vect{x}}
    }.
    \label{eq:eigen-residual}
  \end{equation}
\end{enumerate}

\subsection{Cost}

For dense factors, forming \(\mat{M}=\mat{W}\mat{T}\) costs approximately \(O(mn^2)\), and the reduced eigenvalue problem costs \(O(n^3)\). Storing the factors costs \(O(mn+n^2)\), whereas storing a dense full matrix costs \(O(m^2)\). Applying \(\mat{A}\) to one vector costs \(O(mn+n^2)\). These estimates explain why the factor form is attractive when \(m\gg n\).

A numerical implementation should still consider scaling and conditioning. Forming \(\mat{W}\mat{T}\) may lose accuracy when the columns of \(\mat{T}\) or rows of \(\mat{W}\) are poorly conditioned. In projection methods, orthonormal or biorthogonal bases are often preferred for this reason \parencite{golub2013matrix,trefethen1997numerical}. For conditioning questions the two are not interchangeable, as \cref{sec:singular-values} shows.

\subsection{MATLAB implementation}

The following function computes the reduced eigenvalues and reconstructs full eigenvectors from \(\mat{C}\).

\begin{lstlisting}[language=Matlab,caption={Reduced spectral analysis in MATLAB.},label={lst:matlab}]
function [lambda, X, residual] = reduced_spectrum(T, S, W)
%REDUCED_SPECTRUM Spectral analysis of A = T*S*W without forming A.

    [m, n] = size(T);
    if ~isequal(size(S), [n, n]) || ~isequal(size(W), [n, m])
        error('Incompatible dimensions for T, S, and W.');
    end

    M = W*T;
    C = S*M;
    [V, D] = eig(C);
    lambda = diag(D);
    X = T*V;

    residual = zeros(n, 1);
    for j = 1:n
        x = X(:, j);
        if norm(x) <= eps*norm(T, 'fro')*norm(V(:, j))
            residual(j) = NaN;
        else
            Ax = T*(S*(W*x));
            residual(j) = norm(Ax - lambda(j)*x)/norm(x);
        end
    end
end
\end{lstlisting}

Only eigenvectors associated with nonzero eigenvalues are guaranteed to lift to nonzero full vectors. If \(\mat{C}\) has a zero eigenvalue, the corresponding column of \(\mat{T}\mat{V}\) must be checked separately. The residual is deliberately computed in the full space. A reduced eigenvalue problem can be solved to high accuracy while the lifted vector remains a poor eigenvector of \(\mat{A}\), and by \cref{sec:singular-values} the reduced matrices cannot detect that on their own.

\clearpage
\subsection{Python implementation}

The same computation in NumPy. The two listings are interchangeable, so a reader may work from whichever is closer to hand.

\begin{lstlisting}[language=Python,caption={Reduced spectral analysis with NumPy.},label={lst:python}]
import numpy as np


def reduced_spectrum(T: np.ndarray,
                     S: np.ndarray,
                     W: np.ndarray
                     ) -> tuple[np.ndarray, np.ndarray, np.ndarray]:
    """Analyze A = T @ S @ W without forming the full matrix A."""
    m, n = T.shape
    if S.shape != (n, n) or W.shape != (n, m):
        raise ValueError("Incompatible dimensions for T, S, and W")

    M = W @ T
    C = S @ M
    eigenvalues, V = np.linalg.eig(C)
    X = T @ V

    residuals = np.empty(n, dtype=float)
    for j in range(n):
        x = X[:, j]
        nx = np.linalg.norm(x)
        # np.linalg.norm defaults to the Frobenius norm on a 2-D array
        # and to the 2-norm on a 1-D array, matching norm(T,'fro')
        # and norm(V(:,j)) in the MATLAB listing above.
        if nx <= np.finfo(float).eps*np.linalg.norm(T)*np.linalg.norm(V[:, j]):
            residuals[j] = np.nan
        else:
            Ax = T @ (S @ (W @ x))
            residuals[j] = np.linalg.norm(
                Ax - eigenvalues[j] * x
            ) / nx

    return eigenvalues, X, residuals
\end{lstlisting}

Both listings use only core language features. Neither requires a MATLAB toolbox or any package beyond NumPy, and neither reads external data: the matrices \(\mat{T}\), \(\mat{S}\), and \(\mat{W}\) are supplied by the caller, so the listings are reproducible as printed.

\section{Limitations and cautions}
\label{sec:limitations}

Sylvester's theorem is exact, but the conclusions it supports are specific. Several common overextensions should be avoided.

\subsection{The zero eigenvalue is not fully described}

The nonzero characteristic factors agree exactly. At those eigenvalues more agrees than the characteristic polynomial records. Applying \textcite[Theorem 3.2.11.1]{horn2013matrix} to the two groupings of \cref{sec:three-factors} shows that \(\mat{A}\), \(\mat{B}\), and \(\mat{C}\) carry the same number of Jordan blocks of each size at every nonzero eigenvalue. The zero eigenvalue may have different geometric multiplicity, nilpotent structure, and Jordan block sizes. If long-time behavior depends on generalized zero modes, the reduced characteristic polynomial alone is insufficient. The work of \textcite{flanders1951elementary} gives the classical refined description.

\subsection{Eigenvalues are not singular values}

The singular values of \(\mat{X}\mat{Y}\) and \(\mat{Y}\mat{X}\) generally differ, so the theorem does not preserve operator norms, condition numbers, energy amplification, or least-squares sensitivity. \Cref{sec:singular-values} develops this point in full, including the one case in which singular values do transfer.

\subsection{Equal eigenvalues do not imply equal transient behavior}

A nonnormal matrix can exhibit substantial transient growth even when every eigenvalue lies in the left half-plane. Two matrices with the same eigenvalues can have very different eigenvectors, pseudospectra, and short-time amplification. A stability conclusion based on eigenvalues should therefore be separated from a claim about transient response. Because \(\mat{A}\), \(\mat{B}\), and \(\mat{C}\) need not share singular values, the reduced matrices carry no guarantee here either. Standard numerical linear algebra references discuss conditioning and nonnormality in greater depth \parencite{trefethen1997numerical,golub2013matrix}.

\subsection{Lifting can amplify errors}

Suppose an approximate reduced eigenvector \(\widehat{\vect{v}}\) is lifted as
\(\widehat{\vect{x}}=\mat{T}\widehat{\vect{v}}\). If \(\norm{\mat{T}}\) is large or \(\mat{T}\) is poorly conditioned, small reduced errors may produce large full-space errors. Similarly, an extraction map \(\mat{W}\) of large norm can magnify measurement or rounding errors. Conditioning is the relevant quantity when a map is inverted, and norm is the relevant one for direct amplification.

\subsection{Rank deficiency can reduce the active dimension further}

Although \(n\) is the nominal reduced dimension,
\begin{equation}
  \rank(\mat{A})
  \leq
  \min\{\rank(\mat{T}),\rank(\mat{S}),\rank(\mat{W})\}.
  \label{eq:rank-further}
\end{equation}
If any factor is rank deficient, the true active dimension may be smaller than \(n\), and \(\mat{B}\) and \(\mat{C}\) will have additional zero eigenvalues.

\subsection{A reduced model must be judged by its purpose}

In an approximate projection method, preserving selected eigenvalues may be important, but it may not be the primary objective. A reduced model might instead be designed to reproduce an output, a transfer function, a time interval, a parameter range, or a conserved quantity. Sylvester's theorem explains the internal spectrum of the lifted reduced operator; it does not replace application-specific validation.

\section{Exercises}
\label{sec:exercises}

The exercises are ordered roughly from direct verification to open-ended computation.

\begin{exercise}[Dimensions and rank]
\label{ex:dimensions-rank}
Let \(\mat{T}\in\F^{m\times n}\), \(\mat{S}\in\F^{n\times n}\), and \(\mat{W}\in\F^{n\times m}\).
\begin{enumerate}[label=(\alph*)]
  \item Verify the dimensions of \(\mat{A}\), \(\mat{B}\), and \(\mat{C}\).
  \item Prove \(\rank(\mat{A})\leq n\).
  \item Prove the two inclusions in \cref{eq:range-null-inclusions}.
\end{enumerate}
\end{exercise}

\begin{exercise}[A proof by block elimination]
\label{ex:block-proof}
Starting with the block matrix in \cref{eq:block-K}, multiply it on the left or right by suitable block triangular matrices with determinant one. Derive \cref{eq:sylvester-determinant} without explicitly invoking the Schur-complement formulas.
\end{exercise}

\begin{exercise}[Eigenvector transfer]
\label{ex:eigenvector-transfer}
Let \(\mat{X}\in\F^{m\times n}\) and \(\mat{Y}\in\F^{n\times m}\). Prove that the maps
\[
  \vect{v}\mapsto\mat{X}\vect{v},
  \qquad
  \vect{u}\mapsto\frac{1}{\lambda}\mat{Y}\vect{u}
\]
are inverse isomorphisms between the eigenspaces of \(\mat{Y}\mat{X}\) and \(\mat{X}\mat{Y}\) associated with a fixed nonzero eigenvalue \(\lambda\).
\end{exercise}

\begin{exercise}[The worked matrix example]
\label{ex:verify-worked}
For the matrices in \cref{eq:example1-factors}:
\begin{enumerate}[label=(\alph*)]
  \item compute \(\mat{M}\), \(\mat{B}\), \(\mat{C}\), and \(\mat{A}\);
  \item verify \cref{eq:example1-charA,eq:example1-charBC};
  \item find bases for the eigenspaces of \(\mat{A}\);
  \item identify which eigenvectors lie in \(\range(\mat{T})\) and which lie in \(\nullsp(\mat{W})\).
\end{enumerate}
\end{exercise}

\begin{exercise}[Similarity of the reduced matrices]
\label{ex:similarity}
Prove \cref{prop:B-C-similar}. Give an example in which both \(\mat{S}\) and \(\mat{W}\mat{T}\) are singular and \(\mat{B}\) and \(\mat{C}\) have the same characteristic polynomial but are not similar.
\end{exercise}

\begin{exercise}[Powers]
\label{ex:powers}
Prove \cref{prop:powers} directly for \(k=2\) and \(k=3\), and then give an induction proof. Use the result to derive a formula for
\[
  p(\mat{A})\vect{x}
\]
when \(p\) is a polynomial.
\end{exercise}

\begin{exercise}[A shifted solve]
\label{ex:shifted-solve-exercise}
Repeat \cref{ex:shifted-solve} with
\[
  \mu=2,
  \qquad
  \vect{b}=
  \begin{pmatrix}0\\1\\1\end{pmatrix}.
\]
Solve the reduced system, reconstruct \(\vect{x}\), and verify the answer by direct multiplication.
\end{exercise}

\begin{exercise}[Matrix exponential]
\label{ex:matrix-exponential}
Starting from the power series for \(\e^{t\mat{A}}\), prove \cref{eq:exponential-reduction}. Show that differentiating the formula gives
\[
  \frac{\dd}{\dd t}\e^{t\mat{A}}
  =\mat{A}\e^{t\mat{A}}.
\]
\end{exercise}

\begin{exercise}[Observable reduction]
\label{ex:observable-reduction}
For the full ODE \(\vect{x}'=\mat{A}\vect{x}\), define \(\vect{q}=\mat{W}\vect{x}\).
\begin{enumerate}[label=(\alph*)]
  \item Derive \(\vect{q}'=\mat{B}\vect{q}\).
  \item Show that \(\vect{x}(t)\) can be reconstructed from \(\vect{q}(t)\) using \cref{eq:x-reconstruction-integral}.
  \item Explain why this reduction is valid even when \(\vect{x}_0\notin\range(\mat{T})\).
\end{enumerate}
\end{exercise}

\begin{exercise}[Coordinate reduction]
\label{ex:coordinate-reduction}
Assume \(\mat{T}\) has full column rank and \(\vect{x}_0\in\range(\mat{T})\). Prove that the solution remains in \(\range(\mat{T})\) and derive \(\vect{y}'=\mat{C}\vect{y}\). Why is full column rank needed to conclude the reduced differential equation from
\(\mat{T}\vect{y}'=\mat{T}\mat{C}\vect{y}\)?
\end{exercise}

\begin{exercise}[The ODE example]
\label{ex:verify-ode-example}
For \cref{ex:ode-three-to-two}, take \(a=2\) and \(b=-1\).
\begin{enumerate}[label=(\alph*)]
  \item Write the reduced initial-value problem.
  \item Compute \(\vect{y}(t)\) and \(\vect{x}(t)\).
  \item Verify the initial condition and the differential equation.
  \item Determine the limit of \(\vect{x}(t)\) as \(t\to\infty\).
\end{enumerate}
\end{exercise}

\begin{exercise}[Shifted stability]
\label{ex:shifted-stability-exercise}
Let the eigenvalues of \(\mat{B}\) be
\[
  1+2\mathrm{i},\quad 1-2\mathrm{i},\quad -3.
\]
For which positive values of \(\gamma\) is
\[
  \vect{x}'=(-\gamma\I_m+\mat{T}\mat{S}\mat{W})\vect{x}
\]
asymptotically stable?
\end{exercise}

\begin{exercise}[Projection and invariance]
\label{ex:projection-invariance}
Let \(\mat{L}\in\R^{4\times4}\), and let the columns of
\(\mat{T}\in\R^{4\times2}\) be orthonormal. Set
\(\mat{W}=\mat{T}^{\mathsf T}\) and
\(\mat{L}_r=\mat{W}\mat{L}\mat{T}\).
\begin{enumerate}[label=(\alph*)]
  \item Prove that \(\mat{T}\mat{L}_r\mat{W}\) has nonzero eigenvalues equal to the nonzero eigenvalues of \(\mat{L}_r\).
  \item Show that this does not imply that the eigenvalues of \(\mat{L}_r\) are eigenvalues of \(\mat{L}\).
  \item Find a condition under which they are eigenvalues of \(\mat{L}\).
\end{enumerate}
\end{exercise}

\begin{exercise}[Singular values]
\label{ex:singular-values}
Let \(\mat{X}\) and \(\mat{Y}\) be square matrices of the same size.
\begin{enumerate}[label=(\alph*)]
  \item show that \(\mat{X}\mat{Y}\) and \(\mat{Y}\mat{X}\) have the same determinant, so their singular values have the same product, recalling that this product equals \(\lvert\det\rvert\) rather than \(\det\);
  \item find \(2\times2\) matrices \(\mat{X}\) and \(\mat{Y}\) for which \(\mat{X}\mat{Y}\) and \(\mat{Y}\mat{X}\) have the same eigenvalues but condition numbers differing by a factor of at least ten;
  \item explain why part (b) does not contradict Sylvester's theorem.
\end{enumerate}
\end{exercise}

\begin{exercise}[Computational project]
\label{ex:computational-project}
Generate random matrices with \(m=1000\) and \(n=10\). Compare:
\begin{enumerate}[label=(\alph*)]
  \item the time and memory required to form \(\mat{A}=\mat{T}\mat{S}\mat{W}\);
  \item the time to compute the eigenvalues of \(\mat{A}\);
  \item the time to form \(\mat{B}\) and compute its eigenvalues;
  \item the residuals of lifted eigenvectors.
\end{enumerate}
Repeat with increasingly ill-conditioned \(\mat{T}\), and report how the residuals change.
\end{exercise}

\section{Conclusion}
\label{sec:conclusion}

The four questions of \Cref{sec:introduction} now have answers. The nonzero eigenvalues of the large matrix are exactly those of either small matrix (\Cref{sec:sylvester}, \Cref{sec:three-factors}), and the extracted variable \(\mat{W}\vect{x}\) evolves without reference to the full state (\Cref{sec:ode-reduction}). A shifted large system reduces to an \(n\times n\) solve (\Cref{sec:algebraic-systems}). The same algebra describes the lifted operator of a projection method, though not whether that operator approximates the original (\Cref{sec:projection}).

The factorization
\[
  \mat{A}=\mat{T}\mat{S}\mat{W}
\]
reveals that a large operator acts through an \(n\)-dimensional channel. Sylvester's theorem converts this geometric fact into an exact spectral statement:
\[
  \det(\lambda\I_m-\mat{A})
  =\lambda^{m-n}
   \det(\lambda\I_n-\mat{B})
  =\lambda^{m-n}
   \det(\lambda\I_n-\mat{C}).
\]
The two small matrices have complementary meanings. The matrix
\(\mat{B}=(\mat{W}\mat{T})\mat{S}\) governs the extracted observables
\(\vect{q}=\mat{W}\vect{x}\), while
\(\mat{C}=\mat{S}(\mat{W}\mat{T})\) governs coefficients in the lifted subspace
\(\vect{x}=\mat{T}\vect{y}\). The intertwining identities
\[
  \mat{W}\mat{A}=\mat{B}\mat{W},
  \qquad
  \mat{A}\mat{T}=\mat{T}\mat{C}
\]
make these interpretations precise.

For exact factorizations, the reduction is not heuristic. It gives exact nonzero eigenvalues, exact reduced ODEs, exact formulas for powers and analytic matrix functions, and exact low-dimensional solutions of shifted linear systems. What it does not give is equally definite. Singular values, condition numbers, and transient behavior are not determined by the reduced matrices, and have to be established from the factors themselves. For projection-based model reduction, the same theorem explains the relation between a reduced matrix and its lifted full-space representation. The separate question of whether that representation approximates the original full operator must be answered by the quality of the trial and test spaces, residual analysis, and application-specific validation.

Two directions are left open. The fine structure at the zero eigenvalue, described by \textcite{flanders1951elementary}, is not developed here, and neither is a quantitative bound relating the invariance residual of \cref{eq:invariance-residual} to the accuracy of the reduced spectrum. Both are within reach of a reader who has worked the exercises.

The practical principle is simple, and so is its boundary:
\begin{quote}
When a large linear system acts through a small intermediate space, analyze its spectrum and its exact trajectories inside that space, and use the lifting and extraction maps to connect the reduced results to the full variables. Norms, conditioning, and transient growth must be read from the maps themselves.
\end{quote}

\appendix
\crefalias{section}{appendix}
\crefname{appendix}{appendix}{appendices}
\Crefname{appendix}{Appendix}{Appendices}

\section{Selected solutions and hints}
\label{app:solutions}

\subsection*{Solution to \Cref{ex:dimensions-rank}}

The products have dimensions
\[
  \mat{A}: (m\times n)(n\times n)(n\times m)=m\times m,
\]
\[
  \mat{B}: (n\times m)(m\times n)(n\times n)=n\times n,
\]
and
\[
  \mat{C}: (n\times n)(n\times m)(m\times n)=n\times n.
\]
The rank inequality gives
\[
  \rank(\mat{A})
  \leq\rank(\mat{T})
  \leq n.
\]
For every \(\vect{x}\),
\(\mat{A}\vect{x}=\mat{T}(\mat{S}\mat{W}\vect{x})\), so
\(\mat{A}\vect{x}\in\range(\mat{T})\). If \(\mat{W}\vect{x}=\vect{0}\), then
\(\mat{A}\vect{x}=\mat{T}\mat{S}\vect{0}=\vect{0}\).

\subsection*{Hint for \Cref{ex:block-proof}}

Multiply \(\mat{K}\) from \cref{eq:block-K} by
\[
  \begin{pmatrix}
    \I_m&-\mat{X}\\
    \mat{0}&\I_n
  \end{pmatrix}
\]
on the right to obtain a block triangular matrix involving
\(\I_n+\mat{Y}\mat{X}\). Use a different triangular factor to obtain
\(\I_m+\mat{X}\mat{Y}\). Both triangular multipliers have determinant one.

\subsection*{Solution to \Cref{ex:eigenvector-transfer}}

Let \(E_{YX}(\lambda)\) and \(E_{XY}(\lambda)\) denote the eigenspaces. By
\cref{prop:eigenvector-transfer}, \(\mat{X}\) maps
\(E_{YX}(\lambda)\) into \(E_{XY}(\lambda)\). For
\(\vect{v}\in E_{YX}(\lambda)\),
\[
  \frac{1}{\lambda}\mat{Y}(\mat{X}\vect{v})
  =\frac{1}{\lambda}\mat{Y}\mat{X}\vect{v}
  =\vect{v}.
\]
Similarly, for \(\vect{u}\in E_{XY}(\lambda)\),
\[
  \mat{X}\left(\frac{1}{\lambda}\mat{Y}\vect{u}\right)
  =\frac{1}{\lambda}\mat{X}\mat{Y}\vect{u}
  =\vect{u}.
\]
The two maps are inverses.

\subsection*{Solution to \Cref{ex:shifted-solve-exercise}}

For the matrices in \cref{ex:three-by-three},
\[
  \mat{W}\vect{b}
  =
  \begin{pmatrix}0\\2\end{pmatrix},
  \qquad
  2\I_2-\mat{B}
  =
  \begin{pmatrix}
    4&0\\
    2&8
  \end{pmatrix}.
\]
Solving gives
\[
  \vect{z}
  =
  \begin{pmatrix}0\\1/4\end{pmatrix}.
\]
Using \cref{eq:reconstruct-shifted-solve},
\[
  \vect{x}
  =\frac12\vect{b}
   +\frac12\mat{T}\mat{S}\vect{z}
  =
  \begin{pmatrix}
    0\\1/8\\1/8
  \end{pmatrix}.
\]
A direct calculation confirms
\((2\I_3-\mat{A})\vect{x}=\vect{b}\).

\subsection*{Solution to \Cref{ex:verify-ode-example}}

With \(a=2\) and \(b=-1\),
\[
  \vect{y}(0)=
  \begin{pmatrix}2\\-1\end{pmatrix}.
\]
\Cref{eq:ode-example-y2-sol,eq:ode-example-y1-sol} give
\[
  y_2(t)=-\e^{-2t},
  \qquad
  y_1(t)=\e^{-t}+\e^{-2t}.
\]
Hence
\[
  \vect{x}(t)
  =
  \begin{pmatrix}
    \e^{-t}+\e^{-2t}\\
    -\e^{-2t}\\
    \e^{-t}
  \end{pmatrix}.
\]
At \(t=0\), this is \((2,-1,1)^{\mathsf T}\), as required. Every component tends to zero as \(t\to\infty\).

\subsection*{Solution to \Cref{ex:shifted-stability-exercise}}

The largest real part of an eigenvalue of \(\mat{B}\) is \(1\). By
\cref{cor:stability-criterion}, asymptotic stability holds exactly when
\[
  \gamma>1.
\]

\subsection*{Hint for \Cref{ex:projection-invariance}}

Part (a) follows from \cref{thm:three-factor} with
\(\mat{S}=\mat{L}_r\) and \(\mat{W}\mat{T}=\I_2\). For part (b), choose a trial space that is not invariant under \(\mat{L}\). For part (c), use
\(\mat{L}\mat{T}=\mat{T}\mat{L}_r\).

\section{A compact reference sheet}
\label{app:reference}

For
\[
  \mat{A}=\mat{T}\mat{S}\mat{W},
  \qquad
  \mat{B}=(\mat{W}\mat{T})\mat{S},
  \qquad
  \mat{C}=\mat{S}(\mat{W}\mat{T}),
\]
the main identities are:

\begin{align}
  \det(\lambda\I_m-\mat{A})
  &=\lambda^{m-n}\det(\lambda\I_n-\mat{B})
  \notag\\
  &=\lambda^{m-n}\det(\lambda\I_n-\mat{C}),
  \label{eq:reference-char}\\[4pt]
  \mat{A}\mat{T}&=\mat{T}\mat{C},
  \label{eq:reference-AT}\\
  \mat{W}\mat{A}&=\mat{B}\mat{W},
  \label{eq:reference-WA}\\
  \mat{A}^k&=\mat{T}\mat{S}\mat{B}^{k-1}\mat{W},
  \qquad k\geq1,
  \label{eq:reference-powers}\\
  (\mu\I_m-\mat{A})^{-1}
  &=\frac1\mu\I_m
   +\frac1\mu\mat{T}\mat{S}
    (\mu\I_n-\mat{B})^{-1}\mat{W},
  \label{eq:reference-resolvent}\\
  \vect{q}=\mat{W}\vect{x},
  \quad \vect{x}'=\mat{A}\vect{x}
  &\Longrightarrow
  \vect{q}'=\mat{B}\vect{q},
  \label{eq:reference-q}\\
  \vect{x}=\mat{T}\vect{y},
  \quad \vect{x}'=\mat{A}\vect{x}
  &\Longrightarrow
  \vect{y}'=\mat{C}\vect{y},
  \label{eq:reference-y}
\end{align}
where the last implication assumes \(\mat{T}\) has full column rank and the trajectory remains in \(\range(\mat{T})\).

No identity of this kind holds for singular values. In the one important exception, when the columns of \(\mat{T}\) are orthonormal and \(\mat{W}=\mat{T}^{*}\), the nonzero singular values of \(\mat{A}\) are exactly those of \(\mat{S}\). Without that hypothesis the singular values of \(\mat{A}\), \(\mat{B}\), and \(\mat{C}\) can differ, and so can their condition numbers. See \cref{sec:singular-values}.

\section*{Funding and competing interests}

This work received no external funding. The author declares no competing interests.

\printbibliography

\end{document}